\documentclass[11pt]{amsart}
\usepackage[utf8]{inputenc}
\usepackage{amsmath,amssymb}
\usepackage{centernot}
\usepackage{bbm}
\usepackage{mleftright}
\usepackage{mathtools}
\usepackage[margin=3cm]{geometry}

\usepackage[hidelinks]{hyperref}
\usepackage{appendix}

\usepackage{tikz}
\usepackage{subfigure}

\newtheorem{theorem}{Theorem}[section]
\newtheorem{lemma}[theorem]{Lemma}
\newtheorem{proposition}[theorem]{Proposition}

\theoremstyle{definition}
\newtheorem{definition}[theorem]{Definition}

\newtheorem{example}[theorem]{Example}
\newtheorem{question}[theorem]{Question}
\newtheorem{remark}[theorem]{Remark}

\newcommand{\NN}{{\mathbb{N}}}
\newcommand{\RR}{{\mathbb{R}}}
\newcommand{\ZZ}{{\mathbb{Z}}}

\newcommand{\eins}{\mathbbm{1}}
\newcommand{\flt}{\mathrm{Flt}}
\newcommand{\hausdorff}{d}
\newcommand{\GL}{\mathrm{GL}}

\DeclareMathOperator{\conv}{conv}
\DeclareMathOperator{\aff}{aff}
\DeclareMathOperator{\spann}{span}
\DeclareMathOperator{\inter}{int}
\DeclareMathOperator{\cl}{cl}
\DeclareMathOperator{\relint}{relint}
\DeclareMathOperator{\wdt}{width}
\DeclareMathOperator{\diam}{diam}
\DeclareMathOperator{\ext}{ext}
\DeclareMathOperator{\expo}{expo}
\DeclareMathOperator{\bd}{bd}
\DeclareMathOperator{\vol}{vol}
\DeclareMathOperator{\Vol}{Vol}
\DeclareMathOperator{\centroid}{g}

\newcommand{\sukz}{\lambda}

\newcommand{\normalfan}{\mathrm{N}}
\DeclareMathOperator{\ncone}{ncone}
\newcommand{\facefan}{\mathrm{F}}
\DeclareMathOperator{\cone}{cone}

\title{Lattice Reduced and Complete Convex Bodies}
\author{Giulia Codenotti and Ansgar Freyer}
\date{\today}

\address{Freie Universit\"at Berlin, Fachbereich Mathematik und Informatik, Arnimallee 2, D-14195 Berlin}
\email{giulia.codenotti@fu-berlin.de}
\address{Technische Universit\"at Wien, Institut f\"ur Diskrete Mathematik und Geometrie, Wiedner Hauptstraße 8-10/1046, A-1040 Wien}
\email{ansgar.freyer@tuwien.ac.at} 

\usepackage[colorinlistoftodos]{todonotes}

\numberwithin{equation}{section}
\mathtoolsset{showonlyrefs=true}

\begin{document}

\begin{abstract}
% Following the established terms in Euclidean geometry, we define reduced and complete convex bodies $C\subset\RR^d$ when the width and diameter are measured with respect to a lattice $\Lambda\subset\RR^d$. We prove that the reduced (complete) bodies in this theory are polytopes with a bounded number of vertices (facets) and that their width (diameter) is realized by at least $O(\log d)$ independent directions. Moreover, we present various methods of constructing such polytopes and relate lattice reduced simplices to the study of the flatness constant in integer programming.
The purpose of this paper is to study convex bodies $C$ for which there exists no convex body $C^\prime\subsetneq C$ of the same lattice width. Such bodies shall be called ``lattice reduced'', and they occur naturally in the study of the flatness constant in integer programming, as well as other problems related to lattice width. We show that any simplex that realizes the flatness constant must be lattice reduced and prove structural properties of general lattice reduced convex bodies: they are polytopes with at most $2^{d+1}-2$ vertices and their lattice width is attained by at least $\Omega(\log d)$ independent directions.
Strongly related to lattice reduced bodies are the ``lattice complete bodies'', which are convex bodies $C$ for which there exists no $C^\prime\supsetneq C$ such that $C^\prime$ has the same lattice diameter as $C$. Similar structural results are obtained for lattice complete bodies. Moreover, various construction methods for lattice reduced and complete convex bodies are presented.
\end{abstract}

%%%%
\maketitle
%%%%

\noindent {\bf Mathematics Subject Classification (2020):} 52B05, 52B11, 52B12, 52C07, 52C45. \\
\noindent {\bf Keywords:} \textit{Reduced bodies, complete bodies, lattice width, flatness constant}.

\section{Introduction}
\label{sec:intro}

A \emph{convex body} $C\subset\RR^d$ is a convex compact set with non-empty interior. 
A \emph{lattice} $\Lambda\subset\RR^d$ is a discrete subgroup of $\RR^d$, which we shall assume to be full-dimensional. The \emph{dual lattice} of $\Lambda$ is then defined as $\Lambda^\star = \{y\in\RR^d : x\cdot y \in\ZZ,~\forall x\in \Lambda\}$. The \emph{lattice width} of a convex body $C$ with respect to the lattice $\Lambda$ is defined as
\begin{equation}
\label{eq:lattice_width}
\wdt_\Lambda(C) = \min_{y\in\Lambda^\star\setminus\{0\}}\max_{a,b\in C} y\cdot (a-b).
\end{equation}  
Approximately we can think of the lattice width as the minimum number of parallel lattice hyperplanes which intersect the convex body.
We call convex bodies which do not contain lattice points in their interior \emph{hollow}. The lattice width turns out to be the right parameter to formalize the intuition that hollow convex bodies cannot be too ``large". Indeed, Kinchin's celebrated flatness theorem \cite{khinchine} states that in fixed dimension the lattice width of hollow convex bodies is bounded: there exists a number $c_d>0$ depending only on the ambient dimension $d$ such that for any hollow convex body $C\subset\RR^d$ one has $
\wdt_{\Lambda}(C) \leq c_d$.
The smallest number $c_d$ for which this holds is known as the \emph{flatness constant} and is denoted by $\flt(d)$ in the following. 

The flatness theorem had a great impact on the theory of integer programming, since it was a key ingredient in Lenstra's polynomial time algorithm to solve integer programs for a fixed number of variables \cite{lenstra}. The estimates of the efficiency of these algorithms depend on the best known upper bound for the flatness constant $\flt(d)$. Thus, after Lenstra's discovery there have been many improvements on the upper bound for the flatness constant. In the seminal paper \cite{coveringminima}, Kannan and Lov\'asz show the polynomial bound $\flt(d)\leq cd^2$, where $c$ is an absolute constant. The next breakthrough was achieved by Banaszczyk \cite{transf}, who showed that the lattice width of a  hollow centrally symmetric convex body is at most $cd\log d$. Using this result, the bound of Kannan and Lov\'asz was improved successively by Banasczcyk et al.\ \cite{BanaszczykLitvakPajorEtAl1999}, Rudelson \cite{rudelson} and, very recenty, by Reis and Rothvoss who showed $\flt(d)\leq c d \log(d)^3$ \cite{ReisRothvoss}.
%(Kinchin's original proof yielded a super-exponential bound) 
This bound is nearly optimal, since the following trivial lower bound is linear in $d$: Consider the standard simplex $\Delta_d := \conv(0, e_1,\ldots, e_d)$, where  $e_1,\ldots,e_d$ are the standard basis vectors. Its dilation $d \cdot \Delta_d$ by factor $d$ is a hollow body of lattice width $d$, which implies $\flt(d)\geq d$. 

The simplex $d\cdot \Delta_d$ is however not the hollow convex body with largest width. In fact, determining the \emph{exact} flatness constant is a challenge for each fixed dimension $d >1$. The only known value is in dimension $2$: Hurkens proved that $\flt(2) = 1+\tfrac{2}{\sqrt{3}}$ \cite {hurkens} by finding the (unique!) realizer for the flatness constant via an exhaustive search. A \emph{realizer} for the flatness constant $\flt(d)$ is a hollow convex $d$-body of maximum possible width, that is, equal to the flatness constant. 
All dimensions higher than $2$ remain open, and partial progress can be summarized as follows. In dimension 3, Codenotti and Santos \textbf{\cite{codenottisantos}} constructed a hollow tetrahedron which they conjecture to be a realizer for $\flt(3)$, and Averkov, Codenotti, Macchia and Santos showed in \textbf{\cite{averkovcodenotti}} that this tetrahedron is (at least) a strict local maximizer. %The proof approach invokes criteria of local optimality on the algebraic parametrization of hollow tetrahedra and the semialgebraic description of the lattice width. That is, the problem is reformulated into a (local) polynomial optimization problem. \SPP{NLO}
In dimensions 4 and 5, local maximizers were found by Mayrhofer, Schade and Weltge \cite{weltgeetal} by following up on an approach established in \textbf{\cite{averkovcodenotti}}. They also construct a family of hollow convex bodies $C_d\subset\RR^d$ with $\wdt_{\ZZ^d}(C_d) \geq 2d - o(1)$. 
%, where $o(1)$ stands for a term that goes to $0$, as $d \to \infty$. 

To determine the exact values of $\flt(d)$, a logical first step is to narrow down the search space of potential realizers. One can exploit that the lattice width is a non-decreasing functional under containment and, similarly, that the property of being hollow is preserved under taking subsets. Therefore, Lov\'asz \cite{lovasz_2d} observed that one can restrict the search for realizers of the flatness constant to  \emph{inclusion-maximal hollow bodies}, hollow convex bodies $C$ such that any $K\supsetneq C$ is not hollow. %They were proposed by Lov\'asz \cite{lovasz_hollow} and studied in detail by GA in \textbf{\cite{averkov_lovasz}}.

We propose a novel approach by observing that to study realizers of the flatness constant, one can alternatively restrict to \emph{lattice reduced convex bodies}: 

\begin{definition}
A convex body $C\subset\RR^d$ is called \emph{lattice reduced} with respect to $\Lambda$ if there is no $C^\prime \subsetneq C$ such that $\wdt_\Lambda(C^\prime) = \wdt_\Lambda(C)$. 
\end{definition}

Since the interior lattice point enumerator $C\mapsto |\inter(C)\cap\Lambda|$ is not strictly monotonous, it is not immediate that the flatness constant is achieved by lattice reduced convex bodies. To prove that this does hold, we first have to show that any convex body contains a lattice reduced convex body of the same lattice width (cf.\ Theorem \ref{thm:zorn}). Even so, it is possible that the flatness constant is also realized by non-reduced convex bodies. 
For simplices, we exclude this possibility: 

\begin{theorem}
\label{thm:local}
Let $\Lambda\subset\RR^d$ be a $d$-dimensional lattice and let $S$ be a local maximum of $\wdt_\Lambda$ on the class of hollow $d$-simplices. Then, $S$ is lattice reduced with respect to $\Lambda$.
\end{theorem}

Here, the term ``local'' refers to the Hausdorff-distance between compact subsets of $\RR^d$. 
Let $\flt_s(d)$ denote the maximum lattice width of a hollow simplex in $\RR^d$. According to Theorem \ref{thm:local}, there are no two realizers $S$ and $T$ of $\flt_s(d)$ such that $S\subsetneq T$. Thus, a realizer $T$ of $\flt_s(d)$ is not only reduced, but also {inclusion-maximal hollow}.  Lov\'asz observed \cite{lovasz_2d} that being inclusion-maximal hollow implies that there exists a lattice point $x_F$ within each facet $F$ of $T$ (see also \cite{Averkov2013}). We prove a similar
criterion for reducedness (cf.\ Proposition \ref{prop:reduced_exposed}): there is a width-realizing direction $y_v\in\ZZ^d$ for each vertex $v$ such that $v$ is the unique vertex maximising the width in this direction. 
Investigating the interplay of the vectors $x_F$ and $y_v$ might yield more restrictive structural results as can be obtained by considering reducedness or inclusion-maximal-hollowness on their own and could lead to a classification of the potential realizers of $\flt_s(d)$. This turns the quest for realizers of $\flt_s(d)$ into an optimization problem. In the plane, this program can be carried out successfully without much effort, as we discuss in Remark \ref{rem:hurkens}.

If Theorem \ref{thm:local} can be extended to \emph{all} convex bodies, then this approach can also be used to tackle $\flt(d)$ itself. However, the question whether any realizer of $\flt(d)$ is necessarily reduced is still open. Studying convex bodies that are simultaneously reduced and inclusion-maximal hollow might also be helpful towards answering this question.

In this paper we have the goal of studying lattice reduced convex bodies, with one eye towards the flatness constant and the other toward comparing these to the analogous Euclidean notion. Indeed,
the definition of lattice reduced owes name and inspiration to the analogue notion of reducedness with respect to \emph{Euclidean width}. The Euclidean width 
%The \emph{Euclidean diameter} $\diam_\RR(C)$ is defined as the length of a longest line segment contained in $C$ and the \emph{Euclidean width} 
$\wdt_\RR(C)$ of a convex body $C$ is defined as the minimum Euclidean distance between two parallel supporting hyperplanes of $C$, i.e.,
\begin{equation}
\label{eq:euclidean_width}
\wdt_\RR(C) = \min_{u\in \mathbb{S}^{d-1}} \max_{a,b\in C} u\cdot (a-b),
\end{equation}
where $\mathbb{S}^{d-1}$ denotes the Euclidean unit sphere centered at the origin in $\RR^d$.
%
% In this setting, $C$ is called \emph{complete} if there is no convex body $C^\prime\supsetneq C$ with $\diam_\RR(C^\prime) = \diam_\RR(C)$. 
Just as in the lattice setup, $C$ is then called \emph{Euclidean reduced} if there is no convex body $C^\prime\subsetneq C$ with $\wdt_\RR(C^\prime)=\wdt_\RR(C)$. 
%
%In the Euclidean world, a notion closely related to reducedness... there is another property studied along with being reduced, completeness. 

In the Euclidean world, reducedness is often studied alongside the notion of \emph{completeness}. 
The \emph{Euclidean diameter} $\diam_\RR(C)$ of a convex body $C$ is defined as the length of a longest line segment contained in $C$, and $C$ is called  \emph{Euclidean complete} if there is no convex body $C^\prime\supsetneq C$ with $\diam_\RR(C^\prime) = \diam_\RR(C)$. 
It is well-known that $C$ is complete if and only if $C$ is a so-called body of constant width.
%, i.e., the map
% \[
% \mathbb S ^{d-1}\to \RR,\quad u\mapsto \max_{a,b\in C} %u\cdot (a-b)
% \]
% is constant. 
This implies that a Euclidean complete convex body is in particular Euclidean reduced. We refer to the survey \cite{reduced_survey} for a modern overview on reduced convex bodies in the Euclidean setting and to the book \cite{complete_book} for extensive background on complete convex bodies.

%The purpose of this article is to introduce and study a discrete interpretation of these notions, in which the diameter and width are measured with respect to a lattice $\Lambda \subset\RR^d$. 
The strong connection between the notions of reducedness and completeness in the Euclidean setting motivates us to define the analogue notion of a \emph{lattice complete} convex body and to study the properties of lattice reduced and lattice complete convex bodies in parallel. To define lattice completeness, we need a discrete analogue of diameter. We first need to measure the length of segments with respect to the lattice $\Lambda$: We call a segment $I = [a,b]\subset\RR^d$ a \emph{rational segment} with respect to $\Lambda$, if $b-a\in \spann\{x\}$ for some $x\in\Lambda$. If $I$ is a rational segment, then there exists a unique vector $v_I\in \Lambda$, which generates $\Lambda \cap \spann\{b-a\}$ and is a positive multiple of $b-a$. We call $v_I$ the primitive direction of $I$ in $\Lambda$. The \emph{lattice length} of a rational segment $I$ is now defined as
\[
\Vol_1(I) = \frac{|b-a|}{|v_I|}.
\]
Here, $| v |$ denotes the standard Euclidean norm of a vector $v\in\RR^d$.
So in other words, $\Vol_1(I)$ is the Euclidean length of the segment $I$ normalized by the length of its primitive direction. The \emph{lattice diameter} of 
a convex body $C$ with respect to $\Lambda$ can now be defined as
\begin{equation}
\label{eq:lattice_diameter}
\diam_\Lambda(C) = \max\{\Vol_1(I) \colon I\subset C\text{ rational segment}\}.
\end{equation}
The lattice diameter $\diam_\Lambda(C)$ was used in \cite{generalised_flatness} to study generalizations of the flatness constant in the plane and in \cite{iglesias_santos} to classify 4-simplices without interior lattice points. It also occurs in \cite{coveringminima}, where it is expressed in terms of the first minimum of $C-C$ (see also \ Section \ref{sec:background}). 
With the notion of lattice diameter at hand, the definition lattice complete bodies is analogous to the Euclidean case.
\begin{definition}
A convex body $C\subset\RR^d$  is called \emph{lattice complete} with respect to $\Lambda$, if there is no $C^\prime\supsetneq C$ with $\diam_\Lambda(C^\prime) = \diam_\Lambda(C)$.
\end{definition}

In the plane, reduced lattice polygons have been studied independently by Cools and Lemmens in \cite{coolslemmens}, and a slightly different notion of completeness was investigated by B\'ar\'any and F\"{u}redi \cite{baranyfuredi}. 
We also note that reducedness and completeness have been studied in various other non-Euclidean settings, such as spherical geometry and hyperbolic geometry (see, e.g., \cite{spherical_reduced, spherical_complete, sagmeister, leichtweiss, bezdek}), or Minkowski spaces \cite{minkowski_space, minkowski_space2, gonzalez_et_al}.

In Figure \ref{fig:quad}, lattice reduced and complete polygons are illustrated. Note that the square in Figure \ref{fig:quad} (a) shows that, in contrast to the Euclidean setting, completeness does not imply reducedness. For further and higher-dimensional examples, we refer to Section \ref{ssec:ex}.

\begin{figure}[!htb]
    \centering
    \subfigure[The square  is not reduced with respect to $\ZZ^2$, since it contains the dashed diamond, which has the same lattice width. The diamond is lattice reduced.]{
    \includegraphics[width = .28\textwidth]{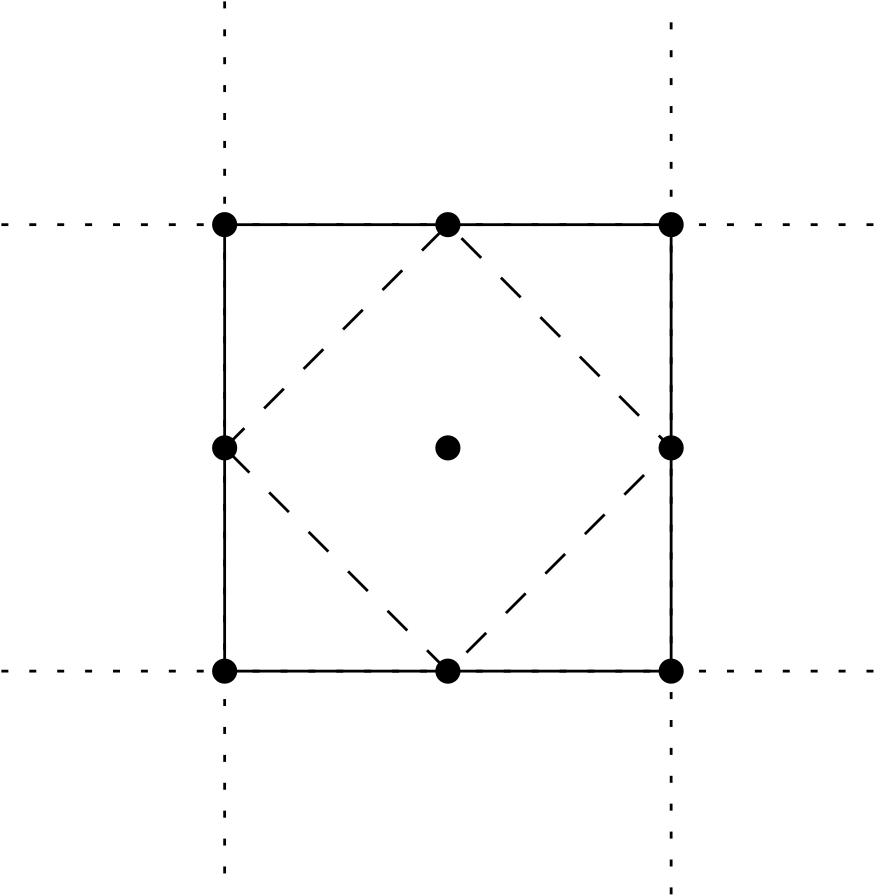}
    }
    \hspace{.02\textwidth}
    \subfigure[The diamond is not complete with respect to $\ZZ^d$, since it is contained in the dashed square, which has the same lattice diameter. The square is lattice complete.]{
    \includegraphics[width = .28\textwidth]{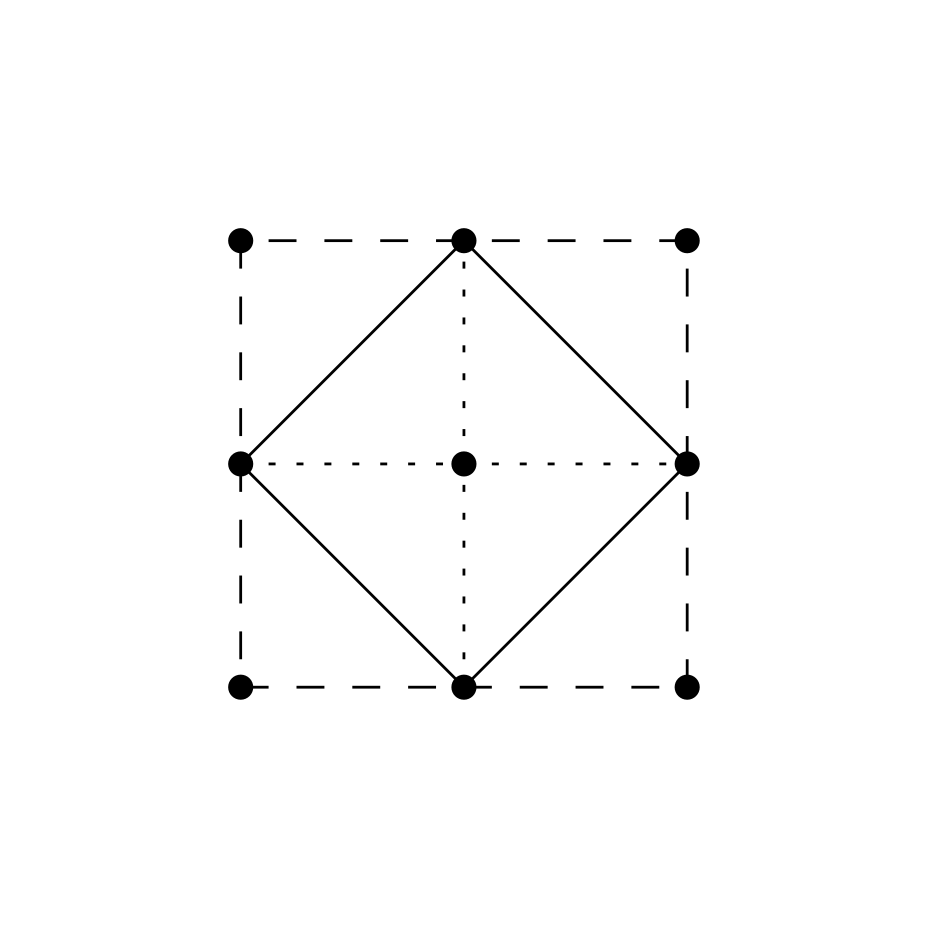}
    }
    \hspace{.02\textwidth}
    \subfigure[The ``skew square'' is simultaneously reduced and complete.]{
    \includegraphics[width = .28\textwidth]{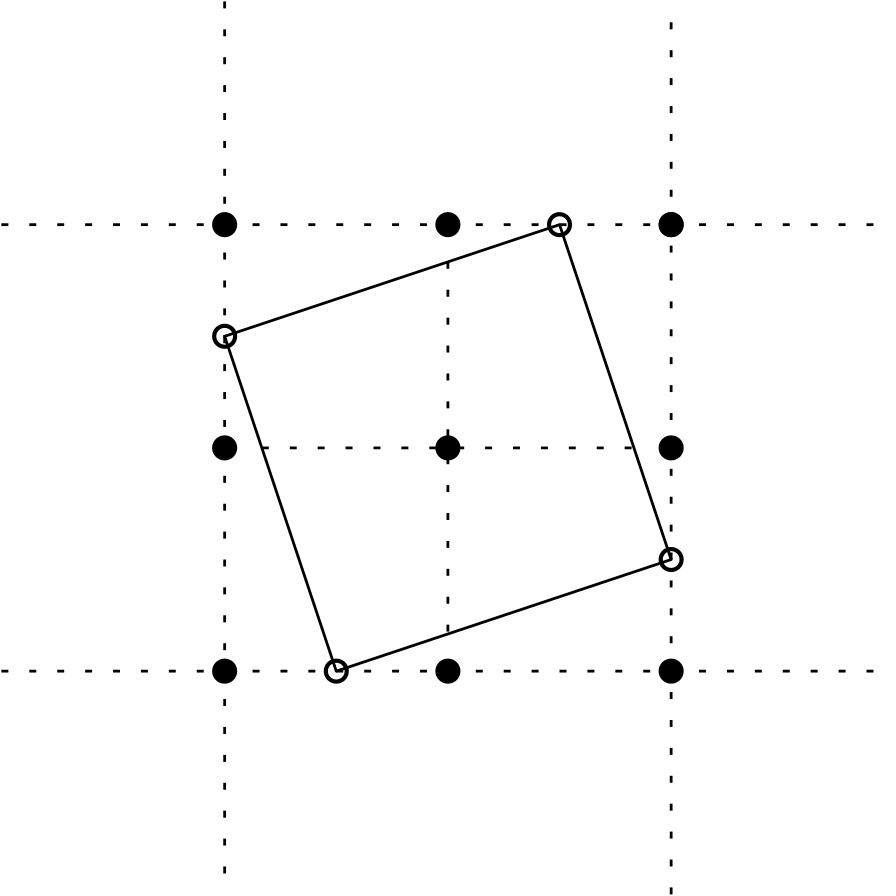}
    }
    \caption{Three examples of reduced or complete quadrilaterals in the plane (cf.\ Example \ref{ex:dim2})}
    \label{fig:quad}
\end{figure}
% In  Proposition \ref{prop:voronoi} we see an important example of a complete convex body with respect to a general $d$-lattice $\Lambda\subset\RR^d$: its Voronoi cell 
% \begin{equation}
% \label{eq:vor_cell}
% V_\Lambda = \{x\in\RR^d : |x| \leq |x-a|,~\forall a \in \Lambda\}.
% \end{equation} 
% Exploiting duality, we also show that the polar body $(V_\Lambda)^\star$ of the Voronoi cell is reduced with respect to $\Lambda^\star$. Here, the \emph{polar body} of a convex body $C\subset\RR^d$ is defined as
% $C^\star = \{y\in\RR^d\colon x\cdot y\leq 1,~\forall x\in C\}$.

% In the non-symmetric case, the simplex $S_d = \conv \{ -\eins,e_1,...,e_d\}$, where $\eins$ denotes the all-1-vector, serves as an example of a convex body which is both reduced and complete (cf.\ Proposition \ref{prop:terminal}). More examples and constructions of complete and reduced convex bodies are given in Section \ref{sec:struct}. 

%Our main results in this paper 
A main focus of the paper concerns the structure of both lattice reduced and lattice complete convex bodies.

\begin{theorem}
\label{thm:polytopes}
Let $\Lambda\subset\RR^d$ be a $d$-dimensional lattice and let $C\subset\RR^d$ be a convex body.
\begin{enumerate}
\item If $C$ is lattice reduced with respect to $\Lambda$, then $C$ is a polytope with at most $2^{d+1}-2$ vertices.
\item If $C$ is lattice complete with respect to $\Lambda$, then $C$ is a polytope with at most $2^{d+1}-2$ facets. 
\end{enumerate}
In both statements, the bound $2^{d+1}-2$ cannot be improved.
\end{theorem}

Indeed, the bound $2^{d+1}-2$ is attained by the Voronoi cell of the lattice $A_d^\star$ and its polar body, respectively.

This result contrasts the Euclidean situation: A Euclidean complete convex body is always strictly convex and thus never a polytope \cite[Proposition 2.1]{oudet}. Euclidean reduced polygons exist for $d=2$, for instance in the form of regular $k$-gons, where $k$ is odd. For $d>2$, constructing a Euclidean reduced polytope is a non-trivial problem. It was solved affirmatively by Gonz\'{a}lez Merino et al.\ in \cite{Merino2018} for $d=3$, but it is open for $d>3$.

Moreover, we give a sharp lower bound on the dimension of the directions that realize the lattice width (diameter) of a lattice reduced (complete) convex body. 

\begin{theorem}
\label{thm:dimension}
Let $\Lambda\subset\RR^d$ be a $d$-dimensional lattice and let $C\subset\RR^d$ be a convex body.
\begin{enumerate}
\item If $C$ is lattice reduced with respect to $\Lambda$, then the vectors $y\in \Lambda^\star$ that realize the minimum $\wdt_\Lambda(C)$ span a linear space of dimension at least $c\cdot\log(d)$, where $c>0$ is a universal constant.
\item If $C$ is lattice complete with respect to $\Lambda$, then the vectors $v\in\Lambda$ for which there exists a rational segment parallel to $v$ achieving $\diam_\Lambda(C)$ span a linear space of dimension at least $c\cdot \log(d)$, where $c>0$ is a universal constant.
\end{enumerate}
In both statements, the order of magnitude $\log(d)$ is best-possible.
\end{theorem}

Theorems \ref{thm:polytopes} and \ref{thm:dimension} suggest that there might be a duality between lattice reduced and lattice complete convex bodies. In Theorem \ref{prop:duality} we show that indeed an origin-symmetric convex body $C$ is lattice reduced with respect to some lattice $\Lambda$ if and only if its polar body $C^\star$ is lattice complete with respect to $\Lambda^\star$. However, there is most likely no such correspondence in the non-symmetric case (cf.\ Remarks \ref{rem:impl1}, \ref{rem:realizers} and  \ref{rem:3simplex}).  

A curious consequence of the complete classification of lattice reduced and complete triangles in the plane up to unimodular equivalence (cf.\ Section \ref{classification_triangles}) is the following theorem.

\begin{theorem}
\label{thm:triangles}
Let $T\subset\RR^2$ be a triangle and let $\Lambda\subset\RR^2$ be a 2-dimensional lattice. If $T$ is lattice complete with respect to $\Lambda$, then $T$ is lattice reduced with respect to $\Lambda$. The converse does not hold.
\end{theorem}

This result echoes the fact that the class of Euclidean complete bodies is properly contained in the class of Euclidean reduced bodies. It further supports the idea that there is no duality between arbitrary lattice reduced and lattice complete convex bodies. We can show however that Theorem \ref{thm:triangles} neither extends to higher-dimensional simplices (cf.\ Remark \ref{rem:3simplex}), nor to general 2-dimensional convex bodies (cf.\ Example \ref{ex:dim2}).

The paper is organized as follows: In Section \ref{sec:background}, we introduce the necessary tools from convex geometry and the geometry of numbers that are needed for our investigations. 
In Section \ref{sec:struct}, we study lattice reduced and complete convex bodies in general. This will lead to the proofs of Theorems \ref{thm:polytopes} and \ref{thm:dimension}, as well as to several examples of reduced and complete convex bodies. Moreover, we discuss the special role of origin-symmetric convex bodies and study the behaviour of reduced and complete convex bodies with respect to standard operations. We finally show that any convex body $C\subset\RR^d$ contains (is contained in) a lattice reduced (complete) body of the same lattice width (diameter). This fact follows rather easily in the Euclidean setting with the help of Zorn's lemma, but in the discrete setting, more care is needed. In Section \ref{sec:simplex}, we turn our eye on reduced and complete simplices, proving Theorems \ref{thm:triangles} and \ref{thm:local}. We then construct a family of complete tetrahedra in Section \ref{sec:simplex} that serves as an interesting source of (counter-)examples. To conclude, in Section \ref{sec:euclid}, we consider lattice analogs of various theorems and conjectures on reduced and complete bodies in the Euclidean theory. The appendix contains a proof for the existence of hollow convex bodies with maximal lattice width.

For the sake of brevity, in the remainder of the paper, the terms ``width'', ``diameter'', ``reduced'' and ``complete'' are to be understood as ``lattice width'' etc, except in Section \ref{sec:euclid} where we always specify lattice- or Euclidean- to avoid confusion.

\tableofcontents

\section{Preliminaries}
\label{sec:background}
For a vector $x\in\RR^d\setminus\{0\}$ we denote by $x^\perp$ the hyperplane orthogonal to $x$ passing through the origin and by $|x|$ its Euclidean length. For two vectors $x,y\in\RR^d$, we let $x\cdot y$ be the standard inner product. The Minkowski sum of two non-empty sets $A,B\subseteq\RR^d$ is denoted by $A+B = \{a+b:a\in A, b\in B\}$ and for a number $\lambda\in\RR$, we write $\lambda A = \{\lambda a: a\in A\}$ and $-A = (-1)A$. The convex hull of a set $A\subset\RR^d$ is denoted by $\conv A$. If $A=\{x,y\}$, we write $[x,y]=\conv A$ for the segment joining $x$ and $y$. The affine hull of $A$ is denoted by $\aff A$. The interior, the closure and the boundary of $A$ with respect to the standard topology in $\RR^d$ are denoted by $\inter A$, $\cl A$ and $\bd A$ respectively and we write, for instance, $\relint A$ for the interior of $A$ with respect to the trace topology on $\aff A$. Moreover, for $n\in\NN$, we write $[n]=\{1,\dots,n\}$. For $i\in [d]$, let $e_i$ denote the $i$th standard basis vector of $\ZZ^d$ and we let $C_d = [-1,1]^d$ be the standard origin-symmetric cube. Finally, we write $\eins = e_1 + \cdots + e_d$ for the all-1-vector.

\subsection{Convex Geometry} 
Here we recall the basic terms and facts from convex geometry that are necessary for the understanding of the following sections. For a thorough read, we refer to the books \cite{gruber_cvx_geom, schneider}, as well as \cite{ziegler} for polytope theory.

A \emph{convex body} $C\subset\RR^d$ means a convex compact set with non-empty interior. A \emph{face} $F\subseteq C$ is a convex subset of $C$ such that for any $x,y\in C$ with $[x,y]\cap F\neq \emptyset$, we have $x,y\in F$. A face is called \emph{proper}, if neither $F$, nor $C\setminus F$ are empty. An \emph{extreme point} is a 0-dimensional face, the set of extreme points of $C$ is denoted by $\ext (C)$.

A hyperplane $H\subset\RR^d$ is said to \emph{support} a convex body $C$, if $C\cap H\neq \emptyset$ and $C$ is contained in one of the two closed halfspaces determined by $H$. If $H$ is a supporting hyperplane of $C$, the set $F=C\cap H$ is called an \emph{exposed face} of $C$. Any exposed face is indeed a face, but the converse is not true. A 0-dimensional exposed face is called \emph{exposed point} and the set of exposed points is denoted by $\expo( C)$. We have $\cl(\expo(C)) = \ext C$.

For a vector $v\in\RR^d$, the \emph{support function} of $C$ is defined as 
$
h(C;v) = \max_{v\in C} x\cdot v.
$
In fact, if $v\neq 0$, the hyperplane $\{x\in\RR^d \colon x\cdot v = h(C;v)\}$ is a supporting hyperplane of $C$.

For a convex set $K\subseteq\RR^d$, its \emph{polar set} is defined as
\[
K^\star =\{y\in\RR^d\colon x\cdot y\leq 1,~\forall x\in K\}.
\]
There is a variety of nice properties for this polarity operation \cite[Section 1.6.1]{schneider}. Here we want to highlight that $K^\star$ is bounded if and only if $0\in\inter K$. In particular, for a convex body $C$ which contains the origin as an interior point, $C^\star$ is a convex body with $0\in\inter C^\star$.

We measure the distance between two convex bodies with respect to the \emph{Hausdorff distance}, i.e., we have
\[
\hausdorff(C,C^\prime) = \min\{ r\geq 0\colon C\subseteq C^\prime + rB\text{ and }C^\prime\subseteq C+rB\},
\]
for any $C,C^\prime\subset \RR^d$, where $B\subset\RR^d$ denotes the Euclidean unit ball. Indeed, $d(\cdot,\cdot)$ turns the set of convex bodies in $\RR^d$ into a metric space.

A \emph{polytope} $P\subset\RR^d$ is
the convex hull of a finite set $A\subset\RR^d$, or, equivalently, a bounded intersection of closed half-spaces.
The extreme points of a polytope are commonly called \emph{vertices} and its $(\dim(P)-1)$-dimensional faces are called \emph{facets}. For a face $F\subseteq P$, the \emph{normal cone} of $F$ is defined as 
\[
\ncone(F;P) = \{y\in\RR^d \colon \forall x\in F, x^\prime\in P,~x\cdot y \geq x^\prime \cdot y\}.
\]
The \emph{normal fan} $\normalfan(P)$ of $P$ is the collection of all normal cones of $P$. $\normalfan(P)$ is a complete fan, i.e., the union of all normal cones of $P$ is $\RR^d$, any face of a normal cone is again a normal cone and any two normal cones intersect in a normal cone.

If $0\in\inter P$, then $P^\star$ is again a polytope and there is a 1-1 correspondence between $k$-faces of $P$ and $(d-1-k)$-faces of $P^\star$; For a $k$-face $F\subseteq P$, the set
\[
F^\diamond = \{y\in P^\star\colon x\cdot y =1,~\forall x\in F\}
\]
is a $(d-k-1)$-face of $P^\star$, the so-called \emph{polar face} of $F$. Here, the empty face is regarded as a $(-1)$-face. Another fan associated to a polytope $P$ with $0\in\inter P$ is the face fan. For a face $F\subseteq P$ one considers the positive hull
\[
\cone(F) = \{ \lambda x \colon x\in F,~\lambda\geq 0\}.
\]
The collection of these cones is the \emph{face fan} $\facefan(P)$. The face fan is connected to the normal fan via polarity; we have $\ncone(F;P) = \cone(F^\diamond)$ and so $\normalfan(P)=\facefan(P^\star)$.

\subsection{Geometry of Numbers} 
In this subsection we summarize tools we need from the geometry of numbers that help describe the interplay of convex bodies and lattices in $\RR^d$. Further background on the geometry of numbers can be found in the books \cite{geometryofnumbers, cassels}.

A \emph{lattice} $\Lambda\subset\RR^d$ is by definition a discrete subgroup of $\RR^d$. Equivalently, we may express $\Lambda$ as the integer span of linearly independent vectors $b_1,\dots,b_k\in\RR^d$, i.e.,
\[
\Lambda = \spann_{\ZZ}\{b_1,\dots,b_k\} = \Big\{\sum_{i=1}^k z_i b_i\colon z_1,\dots,z_k\in\ZZ\Big\}.
\]
The vectors $b_1,\dots,b_k$ are then called a \emph{basis} of $\Lambda$. Two bases  $B=(b_1,\dots,b_k)$ and $B^\prime = (b_1^\prime,\dots,b_k^\prime)$ of $\Lambda$ differ only by a \emph{unimodular transformation}, i.e., $B^\prime = BU$, for a matrix $U\in\ZZ^{k,k}$ with $|\det U |=1$. The group of unimodular transformations on $\RR^d$ is denoted by $\GL_d(\ZZ)$. The \emph{determinant} of a lattice $\Lambda$ can now be defined as $\det\Lambda = \sqrt{\det(B^T B)}$, where $B$ is a basis of $\Lambda$.

Subgroups of $\Lambda$ are called \emph{sublattices}. The \emph{dual lattice} of $\Lambda$ is defined as
\[
\Lambda^\star=\{y\in\spann\Lambda\colon x\cdot y\in\ZZ,~\forall x\in\Lambda\}.
\]
A vector in $x\in\Lambda$ is called \emph{primitive}, if $x\neq 0$ and the segment $[0,x]$ contains no lattice points of $\Lambda$ in its relative interior. More generally, a sublattice $\Lambda^\prime\subseteq\Lambda$ is called primitive, if $\Lambda^\prime = \spann_\RR(\Lambda^\prime)\cap \Lambda$. The primitive vectors $y\in\Lambda^\star$ are in 2-1 correspondence to the primitive $(\dim\Lambda-1)$-dimensional sublattices of $\Lambda$ via the map $y\mapsto \Lambda\cap y^\perp$. The vectors $y$ and $-y$ both lead to the same sublattice.

Three standard examples of lattices that will be of interest in the paper are the integer lattice $\ZZ^d$ and the root lattice $A_d = \{x\in\ZZ^{d+1}\colon x\cdot \eins = 0\}$, as well as its dual $A_d^\star$. We refer to the book \cite{conwaysloanebook} for a detailed treatment of special lattices.

When working with width and diameter, it will be convenient to write $\wdt(C;v) = h(C-C;v)$ for the width of $C$ in direction $v\in\RR^d\setminus\{0\}$. This way, we have $\wdt_\Lambda(C) = \min_y \wdt(C;y)$, where $y$ ranges over $\Lambda^\star\setminus\{0\}$.
Similarly, we write for $v\in\RR^d\setminus\{0\}$
\[
\diam(C;v) = \max\Big\{\frac{\vol_1(I)}{|v|}\colon I\subset C\text{ segment parallel to }v\Big\} 
\]
and find $\diam_\Lambda(C) = \max_v \diam(C;v)$, where $v$ ranges over $\Lambda\setminus\{0\}$.

An important parameter in the geometry of numbers is the first successive minimum of an origin-symmetric convex body $C\subset\RR^d$ with respect to a lattice. It is given by
\[
\sukz_1(C; \Lambda) = \min\big\{\lambda\geq 0\colon (\lambda C)\cap(\Lambda\setminus\{0\})\neq\emptyset\big\} = \max\big\{\lambda \geq 0\colon \inter(\lambda C)\cap \Lambda=\{0\}\big\}.
\]
It is well-known that the width and diameter of a convex body $C$ can be expressed in terms of the first minimum as stated in the following lemma (see, e.g., \cite[Section 1]{coveringminima} and \cite[Section 2.1]{iglesias_santos}). Its proof is a good exercise for anyone that wants to familiarize themselves with the definitions.

\begin{lemma}
\label{lemma:successive_minima}
Let $C\subset\RR^d$ be a convex body and let $\Lambda\subset\RR^d$ be a $d$-lattice. Then we have
\begin{enumerate}
\item $\wdt_\Lambda(C) = \sukz_1((C-C)^\star;\Lambda^\star)$ and
\item $\diam_\Lambda(C) = \sukz_1(C-C;\Lambda)^{-1}$.
\end{enumerate}
Moreover, the directions where the width (resp.\ diameter) of the convex body is acheived correspond to the non-trivial lattice points in $\sukz_1((C-C)^\star;\Lambda^\star)(C-C)^\star$ (resp.\ $\sukz_1(C-C;\Lambda)(C-C)$).
\end{lemma}

We call these directions where the width (diameter) is achieved, i.e. lattice points in $\sukz_1((C-C)^\star;\Lambda^\star)(C-C)^\star$ (resp.\ $\sukz_1(C-C;\Lambda)(C-C)$) simply \emph{width (diameter) directions}. Since the first successive minimum is continuous, it follows immediately that the width and diameter directions of a convex body are stable with respect to small perturbations of $C$. More precisely, the following Lemma holds.
\begin{lemma}
    \label{lemma:continuity}
    Let $C\subset\RR^d$ be a convex body. Then there exists $\varepsilon>0$ such that for any convex body $C^\prime\subset\RR^d$ with $\hausdorff(C,C^\prime)<\varepsilon$ the width (diameter) directions of $C^\prime$ are contained in the width (diameter) directions of $C$.
\end{lemma}

We finish the section by turning to the main focus of the paper, the new definitions of reduced and complete convex bodies, and collecting some first useful properties of these.

\begin{lemma}
\label{lemma:properties}
Let $C\subset\RR^d$ be a convex body and let $\Lambda\subset\RR^d$ be a $d$-lattice.
\begin{enumerate}
\item If $C$ is reduced (complete) with respect to $\Lambda$, then $\lambda C+t$ is reduced (complete) with respect to $\Lambda$, for any $\lambda\in\RR\setminus\{0\}$ and $t\in\RR^d$.
\item If $C$ is reduced (complete) with respect to $\Lambda$, then $AC$ is reduced (complete) with respect to $A\Lambda$, for any $A\in\mathrm{GL}_d(\RR)$.
\item If $C$ is reduced (complete) with respect to $\ZZ^d$, then $UC$ is reduced (complete) with respect to $\ZZ^d$, for any $U\in\mathrm{GL}_d(\ZZ)$.
\end{enumerate}
\end{lemma}

All statements follow directly from the properties of the successive minima. The lemma shows that there is no restriction in considering only the integer lattice $\Lambda=\ZZ^d$. However, when constructing a reduced or complete convex body it is oftentimes more convenient to modify the lattice rather than the convex body, which is why we formulate our results for arbitrary lattices.

\section{Results and constructions for general convex bodies}
\label{sec:struct}

\subsection{Reduced Bodies}
First we show that a convex body is reduced if and only if every exposed point is selected by a width direction. An analog of this statement is also known in the Euclidean case \cite[Theorem 1]{reduced_survey}.
\begin{proposition}
\label{prop:reduced_exposed}
Let $C\subset\RR^d$ be a convex body and let $\Lambda\subset\RR^d$ be a $d$-lattice. The following are equivalent:
\begin{enumerate}
\item $C$ is reduced with respect to $\Lambda$.
\item\label{c:condition} For every exposed point $p\in C$, there exists a width direction $y\in \Lambda^\star$ such that $y\cdot p > y\cdot x$, for all $x\in C\setminus\{p\}$.
\end{enumerate}
\end{proposition}

\begin{proof}
If $C$ is not reduced, there exists a convex body $C' \subsetneq C$ such that $\wdt_\Lambda(C')=\wdt_\Lambda(C)$. Since $C'$ is closed and strictly contained in $C$, there exists a point $p \in \expo(C)$ which is not in $C'$. Thus if condition \eqref{c:condition} would hold, then there would exist a width direction $y\in \Lambda^\star$ such that $\wdt(C'; y) < \wdt(C; y) = \wdt_\Lambda(C)$, a contradiction since $C'$ and $C$ have the same width.

For the converse, let $p\in\expo(C)$ be an exposed point and let $H=\{x: c \cdot x = b\}$, $c\in\RR^d\setminus\{0\}$, be the hyperplane selecting $p$, that is, $H\cap C= \{p\}$ and $C \subset H^-= \{x: c \cdot x \leq b\}$. We define $C_\varepsilon = C \cap H_\varepsilon^-,$ where $H_\varepsilon^- = \{x: c \cdot x \leq b- \varepsilon\}$ and $\varepsilon>0$. Moreover, let $\{y_1, \dots, y_k\}= \wdt_\Lambda(C)(C-C)^\star\cap\Lambda^\star\setminus\{0\}$ be the width directions of $C$, which are finite (cf.\ Lemma \ref{lemma:successive_minima}).

Suppose now that condition \eqref{c:condition} does not hold for $p$, that is, for each of the width directions $y_i$, there exists some $q_i \in C \setminus \{p\}$ such that  $y_i\cdot q_i \geq y_i\cdot p$. We want to show that for $\varepsilon>0$ small enough we have $\wdt_\Lambda(C_\varepsilon) = \wdt_\Lambda(C)$. First observe that, if $\varepsilon$ is chosen sufficiently small, $H_\varepsilon^-$ will still contain $q_1, \dots, q_k$, which implies $q_i \in C_\varepsilon$ and thus $\wdt(C_\varepsilon;y_i) = \wdt(C;y_i)$.

Now we recall from Lemma \ref{lemma:continuity} that for $\varepsilon>0$ small enough, 
the width directions of $C_\varepsilon$ are among the width directions $\{y_1,\dots,y_k\}$ of $C$. In conclusion, for $\varepsilon >0$ small enough it holds that $\wdt_\Lambda(C_\varepsilon)=\wdt_\Lambda(C)$, and thus $C$ is not lattice reduced.
\end{proof}

The next theorem entails the first statements of Theorems \ref{thm:polytopes} and \ref{thm:dimension}.

\begin{theorem}
\label{thm:reduced_meta}
Let $C\subset\RR^d$ be a convex body, reduced with respect to a $d$-lattice $\Lambda$. Let $k=\dim\{y\in\Lambda^\star \colon y\text{ is a width direction of }C\}$. Then, $C$ is a polytope with at most $2^{k+1}-2$ vertices.
\end{theorem}

Before we come to the proof, we need the following lemma.

\begin{lemma}
    \label{lemma:3}
 Let $C\subset\RR^d$ be a convex body with $\inter C\cap \Lambda = \{0\}$. Let $A\subseteq \bd(C)\cap \Lambda$ such that no two points in $A$ lie in a common facet of $C$. Then, $|A|\leq 2^{k+1}-2$, where $k$ is the dimension of $\spann A$. 
\end{lemma}

\begin{proof}
 First we assume that $k=d$. Suppose there are points $a,b\in A$ with $a\neq b$ and $a\equiv b \mod 2\Lambda$. Then $(a+b)/2$ is a lattice point in $\inter C$, because otherwise, $a$ and $b$ had to be contained in a common face of $C$. Consequently, $-a=b$. Thus, there cannot be more than 2 points of $A$ within one coset of $\Lambda/2\Lambda$. Further, there is no point  $a\in A\cap 2\Lambda$, since then, $a/2$ would be a non-zero point in $\inter C$. As there are $2^d$ cosets, the statement follows.

If $k<d$, we consider the convex body $C^\prime = C\cap L$, where $L=\spann A$, and the lattice $\Lambda^\prime = \Lambda\cap L$. Then, $C^\prime$ and $\Lambda^\prime$ are $k$-dimensional in $L$. So the statement follows by applying the full-dimensional case in $L$.
\end{proof}

\begin{proof}[Proof of Theorem \ref{thm:reduced_meta}]
Since there are only finitely many width directions of $C$, it follows from Proposition \ref{prop:reduced_exposed} that $C$ has only finitely many exposed points and is, thus, a polytope.

Let $v_1,\dots,v_n\in C$ be the vertices of $C$ and let $y_i\in \Lambda^\star$ be the width direction which only selects $v_i$, as guaranteed by Proposition \ref{prop:reduced_exposed}, $1\leq i \leq n$. Then, $y_i$ is contained in the interior of the normal cone $\ncone(v_i;C)\in\normalfan(C)$. Since the normal fan of $C-C$ is the common refinement of $\normalfan(C)$ and $\normalfan(-C)$ \cite[Proposition 7.12]{ziegler}, no two distinct vectors $y_i$ and $y_j$ fall in the same normal cone of $C-C$. By duality, this implies that no two of the vectors $y_i$ and $y_j$ are contained in a common face cone in $\facefan((C-C)^\star)$. In particular, it follows that there is no dilation of $(C-C)^\star$ such that $y_i$ and $y_j$ lie in a common facet of this dilation. 

Now consider the polytope $P = \sukz_1((C-C)^\star;\Lambda^\star)(C-C)^\star$. From the definition of the successive minima we have that $P$ contains the origin as its unique interior lattice point. Since the $y_i$ are width directions of $C$, they are contained in $\bd(P)$. Since no two of the vectors $y_i$ and $y_j$ are contained in a common facet of $P$, we can apply Lemma \ref{lemma:3} to $P$ and $A=\{y_1,\dots,y_n\}$ to obtain $n\leq 2^{k+1}-2$ as desired.
\end{proof}

\subsection{Complete Bodies}
Here we prove a similar statement to Theorem \ref{thm:reduced_meta} for complete convex bodies.

\begin{theorem}
\label{thm:complete_meta}
Let $C\subset\RR^d$ be a convex body, complete with respect to a $d$-lattice $\Lambda$. Let $k=\dim\{y\in\Lambda^\star \colon y\text{ is a diameter direction of }C\}$. Then, $C$ is a polytope with at most $2^{k+1}-2$ facets.
\end{theorem}

Although Theorem \ref{thm:complete_meta} is exactly the dual statement to Theorem \ref{thm:reduced_meta} in the origin-symmetric case, we are not aware of an argument that allows to derive one of the theorems from the other. 

Again, we need some preliminary results in order to prove Theorem \ref{thm:complete_meta}.

\begin{lemma}
    \label{lemma:1}
    Let $C$ be a convex body and consider $[a,b],[c,d]\subset C$ such that $b-a$ and $d-c$ are in  a common face of $C-C$. Then there exist opposite exposed faces $F_1,F_2\subset C$ of $C$ with $a,b\in F_1$ and $c,d\in F_2$. Here, $F_1$ and $F_2$ are called \emph{opposite}, if there exist parallel supporting hyperplanes $H_1,H_2\subset\RR^d$ such that $F_i = C\cap H_i$, $i\in\{1,2\}$.
\end{lemma}

\begin{proof}
    Since any face of $C-C$ is contained in an exposed face, there exists a vector $v \in \RR^n\setminus\{0\}$ such that
    \[
        v\cdot (b-a) = v\cdot (d-c) = h(C-C;v) = h(C;v) + h(-C;v).
    \]
    It follows that
    \[
        v\cdot b = v\cdot d = h(C;v)\text{\quad and \quad}v\cdot (-c) = v\cdot (-a) = h(-C;v), 
    \]
    which proves the lemma.
\end{proof}

From this it follows immediately that for a fixed direction $v$, the segments of maximal length in a convex body that are parallel to $v$ pass between two fixed opposite faces.

\begin{lemma}
    \label{lemma:2}
    Let $v\in\RR^d\setminus\{0\}$ and let $S,S^\prime\subset C$ be two segments parallel to $v$ of maximal length in $C$. Then there exists a pair of opposite exposed faces $F,G$ such that $S$ and $S^\prime$ have an endpoint in $F$ and $G$ each.
\end{lemma}

We call \emph{diameter segments} those segments contained in the convex body which have maximum lattice length. The following proposition gives a helpful characterization of complete polytopes.

\begin{proposition}
\label{prop:endpoints}
A $d$-polytope $P\subset\RR^d$ is complete with respect to a $d$-lattice $\Lambda\subset\RR^d$ if and only if for each facet $F\subset P$, there is a diameter segment $I_F\subset P$ with an endpoint in $\relint F$.    
\end{proposition}

\begin{proof}
    If there were a facet not satisfying the condition above, then stacking that facet by adding a point beyond that facet but beneath all other facets
    would give a polytope with the same diameter, since this operation will not change diameter directions and does not lengthen existing diameter segments (the only maximal segments that can be extended in the stacking have an endpoint in $\relint F$).

    If $P$ were not complete but satisfies the condition of the proposition, let $Q$ be a polytope strictly containing it of the same diameter. If $q$ is a point of $Q\setminus P$, then it must lie beyond some facet of $P$, say $F$. Then the segment $I_F$ of the condition of the proposition can be extended to a longer segment in $Q$, contradicting the diameter condition on $Q$. 
 \end{proof}

Now we are well-equipped for the proof of Theorem \ref{thm:complete_meta}.

\begin{proof}[Proof of Theorem \ref{thm:complete_meta}]
    Without loss of generality, suppose $\diam_\Lambda(C)=1$ (cf.\ Lemma \ref{lemma:properties}). In view of Lemma \ref{lemma:successive_minima}, the diameter of $C$ is attained in finitely many directions. By Lemma \ref{lemma:2}, for each such diameter direction $v$, there exist two opposite exposed faces that contain the endpoints of all diameter segments parallel to $v$. So we have finitely many exposed faces $F_1,\dots,F_s\subset C$ that contain the endpoints of all diameter segments. Let $u_1,\dots,u_s$ be the outer unit normal vectors of $C$ that realize the facets $F_i$, that is, $F_i = \{x\in C\colon x\cdot u_i = h(C;u_i)\}$.
    
    Suppose that 
    $C$ is strictly contained in the intersection of the half-spaces $\{x\in\RR^d\colon u_i\cdot x \leq h(C;u_i)\}$. 
    Then there is a point $x\in \bd(C)$ with $u_i \cdot x < h(C;u_i)$, 
    for all $1\leq i \leq s$. 
    Let $\varepsilon>0$ be such that
    $ \varepsilon < h(C;u_i) - u_i\cdot x$, 
    for all $1\leq i\leq s$, and define $x_\varepsilon = x+\varepsilon v\not\in C$, for some outer unit normal vector $v$ of $C$ at $x$. We claim that $C^\prime = \conv(C\cup\{x_\varepsilon\})$ has the same diameter as $C$.

   If $\varepsilon$ is small enough, we can assume that the lattice diameter of $C^\prime$ is attained by one of the diameter directions of $C$ (cf.\ Lemma \ref{lemma:continuity}). Denote this direction by $y$ and let $F_1$ and $F_2$, say, be the faces of $C$ that contain the endpoints of diameter segments in this direction. By the choice of $\varepsilon$, we have
    \[
        x_\varepsilon \cdot u_i =  x\cdot u_i + \varepsilon v\cdot u_i \leq x\cdot u_i+\varepsilon < h(u_i; C),
    \]
    for all $i$. Thus, the hyperplanes $H_i=\{x\in\RR^d : u_i\cdot x = h(u_i;C)\}$ are also supporting hyperplanes of $C^\prime$. Therefore, a segment in $C^\prime$ parallel to $y$ cannot be longer than a segment passing between $F_1$ and $F_2$. So we obtain $\diam_\Lambda(C^\prime) = \diam_\Lambda(C)$, contradicting the completeness of $C$. It follows that $C$ is indeed the intersection of the halfspaces $\{x\in\RR^d\colon u_i\cdot x \leq h(u_i; C)\}$ and as such a polytope.

    In order to see that $C$ has at most $2^{k+1}-2$ facets, we recall from Proposition \ref{prop:endpoints}, that for each facet $F\subset C$, there exists a diameter segment $[x_F, y_F]\subset C$ such that $x_F\in\relint F$. 
    For each facet $F\subset C$, let $v_F = x_F-y_F$. Since $\diam_\Lambda(C) = 1$, these are boundary lattice points of $C-C$. Suppose $v_F$ and $v_G$ are in a common facet of $C-C$. By Lemma \ref{lemma:1}, we have that $x_F$ and $x_G$ are in a common facet of $C$. By construction, we must have $F=G$. Thus, no two distinct vectors $v_F$ share a facet.  Since, due to $\diam_\Lambda(C)=1$, we have $\inter(C-C)\cap\Lambda=\{0\}$, the inequality follows from applying Lemma \ref{lemma:3} to the vectors $v_F$ and $C-C$.
\end{proof}

\subsection{The Origin-Symmetric Case}
For origin-symmetric convex bodies, there is a duality between reducedness and completeness. 

\begin{theorem}
\label{prop:duality}
    Let $C\subset\RR^d$ be an origin-symmetric $d$-polytope and let $\Lambda\subset\RR^d$ be a $d$-lattice. 
    \begin{enumerate}
        \item The width directions of $C$ with respect to $\Lambda$ are the diameter directions of 
        $C^\star$ with respect to $\Lambda^\star$.
        \item $C$ is reduced with respect to $\Lambda$, if and only if $C^\star$ is complete with respect to $\Lambda^\star$. 
    \end{enumerate}
\end{theorem}

\begin{proof}
    For (1), we recall from Lemma \ref{lemma:successive_minima} that the width directions of $C$ with respect to $\Lambda^\star$ are given by 
    \[
    W = \sukz_1((C-C)^\star;\Lambda^\star)(C-C)^\star \cap \Lambda^\star\setminus\{0\}.
    \]
    Since $-C = C$,  we have
    \[\begin{split}
      \sukz_1((C-C)^\star;\Lambda^\star)(C-C)^\star &= \sukz_1\Big( \frac 12 C^\star;\Lambda^\star\Big) \frac 12 C^\star\\
      &=  \sukz_1(2 C^\star;\Lambda^\star) 2 C^\star = \sukz_1(C^\star-C^\star;\Lambda^\star)(C^\star - C^\star),
    \end{split}\]
    where we used that the successive minima are $(-1)$-homogeneous in the first argument and that $C^\star$ is origin-symmetric as well. It follows that
    \[
     W = \sukz_1(C^\star-C^\star;\Lambda^\star)(C^\star - C^\star) \cap \Lambda^\star\setminus\{0\},
    \]
    which proves the claim.

    For (2), we may assume that $C$ is a polytope. Moreover, we can assume that $ \diam_{\Lambda^\star}(C^\star)=2$ (cf.\ Lemma \ref{lemma:properties}). In this case the width directions of $C$ and the diameter directions of $C^\star$ are given by $W=C^\star\cap \Lambda^\star\setminus\{0\}$.
    
    First, let $C^\star$ be complete with respect to $\Lambda^\star$. Consider a vertex $v\in C$ of $C$ and its polar facet $v^\diamond\subset C^\star$. Since $C^\star$ is complete, there exists a diameter direction $y\in W$ and a segment $I=[a,a+2y]\subset C^\star$ such that $a+2y\in\relint v^\diamond$ (cf.\ Proposition \ref{prop:endpoints}). Since $C$ is origin-symmetric, the segment $-I\subset C$ is also diameter realizing and by Lemma \ref{lemma:1} we have $-a\in  v^\diamond$. It follows that $y = \tfrac{1}{2}(a+2y-a)\in\relint v^\diamond$.

    This shows that $y\in\inter(\cone{v^\diamond}) = \inter(\ncone(v;C))$, which means that $y\cdot v > y\cdot x$ for all $x\in C\setminus\{v\}$. Thus, $C$ is reduced with respect to $\Lambda$ by (1) and Proposition \ref{prop:reduced_exposed}. The reverse implication is proved by similar means.
\end{proof}

\begin{remark}\label{rem:non_symmetric_duality}
    There is no direct generalization of Theorem \ref{prop:duality} to arbitrary convex bodies. In fact, if $C$ is not origin-symmetric, $C^\star$ might even be unbounded. Since width and diameter are invariant with respect to translations one could, however, consider $(C-x_0)^\star$ instead, where $x_0$ is a point in $\inter C$. A natural choice for $x_0$ might be the centroid $\centroid(C) = \vol(C)^{-1}\int_C x\,\mathrm dx$. But as we will see in Remarks \ref{rem:impl1} and \ref{rem:3simplex}, we have
    \begin{equation}
    \label{eq:implication1}
    C\text{ is reduced w.r.t.\ $\Lambda$}\centernot\implies~(C-\centroid(C))^\star\text{ is complete w.r.t.\ }\Lambda^\star
    \end{equation}
    and 
    \begin{equation}
    \label{eq:implication2}
    C\text{ is complete w.r.t.\ $\Lambda$}\centernot\implies~(C-\centroid(C))^\star\text{ is reduced w.r.t.\ $\Lambda^\star$}.
    \end{equation}
\end{remark}

\subsection{Constructions and Examples}
\label{ssec:ex}
We start by considering 2-dimensional examples.

\begin{example}
\label{ex:dim2}
In this example, we consider the lattice $\Lambda=\ZZ^2$.
\begin{enumerate}
\item The triangle $T_{\mathrm{orth}}=\conv\{0,e_1,e_2\}$ is reduced, but not complete. Its width is attained by the directions $\pm e_1$, $\pm e_2$ and $\pm \eins$. Each of these directions cut out one of the three vertices of $T_{\mathrm{orth}}$. The diameter of $T_{\mathrm{orth}}$ is attained only by its edges. Thus, $T_{\mathrm{orth}}$ is not complete by Proposition \ref{prop:endpoints}.
\item The square $C_2=[-1,1]^2$ is complete since it is a multiple of the Voronoi cell of $\ZZ^2$ (cf.\ Proposition \ref{prop:voronoi}). It is not reduced since its width is attained by $\pm e_1$ and $\pm e_2$, but these directions support $C_2$ at its edges, not at vertices, thus violating Proposition \ref{prop:reduced_exposed}. By polarity, $C_2^\star=\conv\{\pm e_1,\pm e_2\}$ is reduced, but not complete.

\item Interestingly, the ``twisted squares'' $Q_x = \conv\{\pm\tbinom 1x, \pm\tbinom {x}{-1}\}$ are simultaneously reduced and complete for $x\in (0,1)$; Indeed, both width and diameter are attained by the directions $\pm e_1, \pm e_2$ and, due to the choice of $x$, the corresponding lines support $Q_x$ at its vertices and the corresponding segments pass between two opposite edges (cf.\ Figure \ref{fig:quad} (c)).
\item The pentagon $\conv\{\tbinom{-1}{1},\tbinom 12, \tbinom{2}{0}, \tbinom{0}{-2}, \tbinom {-2}{-1}\}$ is reduced and not complete.
\item The hexagon $H=\conv\{ \pm \tbinom{2}{-1}, \pm\tbinom{-1}{2}, \pm\tbinom 11\}$ is simultaneously reduced and complete.
\end{enumerate}
(4) and (5) can be verified by means of elementary, but rather lengthy, computations.
\end{example}

We continue with our ``standard examples'' for complete and reduced polytopes: The \emph{Voronoi cell} $V_\Lambda$ of a lattice $\Lambda$ is defined as the set of points that are closer to the origin than to any other point in $\Lambda$. Equivalently, we have 
\begin{equation}
    \label{eq:vor_ineq}
    V_\Lambda = \{ x\in\RR^d \colon v\cdot x \leq |v|^2/2,~\forall v\in\Lambda\}.
\end{equation}
Since $\Lambda$ is a discrete set, it follows that $V_\Lambda$ is a polytope. The vectors $v\in\Lambda$ such that the inequality $v\cdot x \leq |v|^2/2$ contributes a facet to $V_\Lambda$ are called \emph{Voronoi relevant}.

\begin{proposition}
\label{prop:voronoi}
Let $\Lambda\subset\RR^d$ be a $d$-lattice. Then, $V_\Lambda$ is complete with respect to $\Lambda$ and $(V_\Lambda)^\star$ is reduced with respect to $\Lambda^\star$.
\end{proposition} 

\begin{proof}
    It is well known (see for instance \cite[Section 4]{horvath}) that $2V_\Lambda$ contains each Voronoi relevant vector $v_F$ in the relative interior of its corresponding facet $F$ and  moreover, we have $\inter(2V_\Lambda)\cap\Lambda=\{0\}$. Since $V_\Lambda$ is origin-symmetric, we have $\sukz_1(V_\Lambda-V_\Lambda;\Lambda) = \sukz_1(2V_\Lambda;\Lambda) = 1$ and, thus, $\diam_\Lambda(V_\Lambda) = 1$. Again by symmetry around the origin, the diameter is attained by the segments $[v_F,-v_F]$. Hence, Proposition \ref{prop:endpoints} applies and $V_\Lambda$ is indeed complete.
    The statement for $(V_\Lambda)^\star$ follows directly from Proposition \ref{prop:duality}.
\end{proof}

An interesting non-symmetric example is given by the simplices \[S_d = \conv\{-\eins, e_1,\dots,e_d\}\subset\RR^d.\]

\begin{proposition}
    \label{prop:terminal}
    The simplices $S_d\subset\RR^d$ are simultaneously reduced and complete with respect to $\ZZ^d$.
\end{proposition}

\begin{proof}
    We start with the reducedness of $S_d$. Since $S_d$ has integer vertices, we have $\wdt_{\ZZ^d}(S_d)\in\ZZ_{>0}$. The vectors $e_i$, $1\leq i \leq d$, show that $\wdt_{\ZZ^d}(S_d)\leq 2$ and the fact that the origin is an interior lattice point of $S_d$ gives $\wdt_{\ZZ^d}(S_d) > 1$. Hence, $\wdt_{\ZZ^d}(S_d) =2$, realized by the standard basis vectors of $\ZZ^d$. The linear functions $x\mapsto x\cdot e_i$ are uniquely maximized by the vertices $e_i\in S_d$ and uniquely minimized by the vertex $-\eins\in S_d$. Therefore, $S_d$ is reduced by Proposition \ref{prop:reduced_exposed}.

    For the completeness of $S_d$ we observe that the facets of $C = S_d-S_d$ are of the form $F=G-G^\prime$, where $G\subset S_d$ is a proper face of $S_d$ and $G^\prime$ is its complementary face, i.e., the convex hull of the vertices of $S_d$ which are not in $F$. So the vertices of $G$ and $G^\prime$ partition $\{-\eins, e_1,\dots,e_d\}$ and we shall assume that $-\eins\in G^\prime$. Let $J\subseteq [d]$ be the set of indices $j$ such that $e_j\in G$ .  A simple computation then shows that
    \[\begin{split}
        \aff F &= \aff\{e_j-e_i,e_j+\eins \colon j\in J,~i\not\in J\}\\
        &= \left\{ x\in \RR^d\colon \frac{d-|J|+1}{d+1} \sum_{j\in J} x_j - \frac{|J|}{d+1}\sum_{j\not\in J}x_j = 1\right\}
    \end{split}\]
    and, thus, 
    \begin{equation}
    \label{eq:sdminussd}
        C =\left \{x\in\RR^d\colon \left| \frac{d-|J|+1}{d+1} \sum_{j\in J} x_j - \frac{|J|}{d+1}\sum_{j\not\in J}x_j \right| \leq 1,~\forall J\subseteq [d],J\neq\emptyset \right\}.
    \end{equation}
    Now we claim that $\diam_{\ZZ^d}(S_d) = 1 + \tfrac 1d$. First, we observe that the segment $I_0 = [-\eins, \tfrac{1}{d}\eins]\subset S_d$ has lattice length $1+\tfrac{1}{d}$, which shows that $\diam_{\ZZ^d}(S_d)\geq 1+\tfrac{1}{d}$, or, in other words, $\sukz_1(C;\ZZ^d)\leq \tfrac{d}{d+1}$.  It, thus, follows from $S_d\subset [-1,1]^d$ that $\diam_{\ZZ^d}(S_d)$ is attained by a non-zero vector $v\in \tfrac{d}{d+1} C \cap \ZZ^d\subset [-1,1]^d$. We may assume that $J=\{j\in [d]\colon v_j=1\}$ is non-empty (otherwise, we consider $-v$ instead). For $j\not\in J$, we have $v_j \leq 0$ and we obtain
    \[
        \frac{d-|J|+1}{d+1}\sum_{j\in J} v_j - \frac{|J|}{d+1}\sum_{j\not\in J}v_j \geq |J|\frac{d-|J|+1}{d+1}\geq \frac{d}{d+1}.
    \]
    So the only non-zero integer points in $\tfrac{d}{d+1}C$ lie in the boundary. This shows $\sukz_1(C;\ZZ^d) = \frac{d}{d+1}$ and thus $\diam_{\ZZ^d}(S_d) = 1+\tfrac{1}{d}$.

    Now it suffices to observe that for any $i\in [d]$, the segments $I_i = [e_i, -\tfrac 1d e_i]$ have lattice length $1+\tfrac 1d$ and are contained in $S_d$; indeed,
    \[
    -\frac 1d e_i = \frac 1d (-\eins + \sum_{j\neq i} e_j) \in\relint F_i,
    \]
    where $F_i\subset S_d$ denotes the facet opposite to $e_i$. Thus, $S_d$ is complete according to Proposition \ref{prop:endpoints}.
\end{proof}

\begin{remark}
    \label{rem:impl1}
    The simplex $S_d$ proves \eqref{eq:implication1}. Its centroid is the origin and its polar $S_d^\star$ is equivalent to the orthogonal simplex $\conv\{0,e_1,\dots,e_d\}$, which is not complete.
\end{remark}

The following classical constructions preserve completeness and reducedness respectively.

\begin{lemma}
\label{lemma:constr}
Let $P\subset \RR^k$ and $Q\subset\RR^l$ be full-dimensional polytopes.
\begin{enumerate}
\item If $P$ and $Q$ are complete with respect to $\ZZ^k$ and $\ZZ^l$ respectively and if $\diam_{\ZZ^k}(P) = \diam_{\ZZ^l}(Q)$, then their Cartesian product $P\times Q$ is complete with respect to $\ZZ^{k+l}$.
\item If $0\in\inter P$, $0\in \inter Q$, and if $P$ and $Q$ are reduced with respect to $\ZZ^k$ and $\ZZ^l$ respectively, such that $\wdt_{\ZZ^k}(P)=\wdt_{\ZZ^l}(Q)$, then their free sum
\[
P \oplus Q = \conv\big( (P\times \{0\}^l)\cup (\{0\}^k\times Q)\big) \subset\RR^{k+l}
\]
is reduced with respect to $\ZZ^{k+l}$.

\item Under the same conditions of (2), the join of height $h$
\[
P \star Q = \conv\big((P\times \{0\}^{l+1})\cup (\{0\}^k\times Q\times \{ h \})\big)\subset\RR^{k+l+1},
\]
where $|h|\geq \wdt_{\ZZ^k}(P)$, is reduced with respect to $\ZZ^{k+l+1}$.

\end{enumerate}
\end{lemma}

\begin{proof}
In the following, let $\diam = \diam_{\ZZ^n}$ and $\wdt = \wdt_{\ZZ^n}$, where the ambient dimension $n$ will be clear from the context.

For (1), we start by proving $\diam(P\times Q)=\diam(P)$. Clearly, any rational segment in $P$ is also a rational segment in $P\times Q$ with the same length and so we have $\diam(P\times Q) \geq \diam(P)$. Conversely, let $I = [(a,a^\prime),(b,b^\prime)]\subset P\times Q$ be a rational segment, where $a,b\in P$ and $a^\prime, b^\prime \in Q$. If $I$ is orthogonal to $\{0\}^k\times \RR^l$, i.e., if $a^\prime = b^\prime$, then $\Vol_1(I) = \Vol_1([a,b])\leq \diam(P)$. If $I$ is not orthogonal to $\{0\}^k\times \RR^l$, then the projection
\[
\Pi: P\times Q \to Q,~(x,y)\mapsto y
\]
restricted to the segment $I$ is injective and maps lattice points to lattice points. It follows that \[\Vol_1(I) \leq \Vol_1(\Pi (I))\leq \diam(Q) = \diam(P).\]
This proves $\diam(P\times Q) = \diam(P)$. In order to see that $P\times Q$ is complete, we note that any facet of $P\times Q$ is of the form $F\times Q$ or $P\times G$, where $F\subset P$ and $G\subset Q$ are facets of $P$ and $Q$, respectively. For a facet of the form $F\times Q$, there exists a point $x\in\relint F$ and a diameter segment $I_F = [x,x^\prime]\subset P$, for some $x^\prime \in P$, since $P$ is complete. Let $y\in\inter Q$ be any interior point of $Q$. The segment $[(x,y), (x^\prime,y)]\subset P\times Q$ is diameter realizing for $P\times Q$ and the point $(x,y)$ is in the relative interior of $F\times Q$. A facet of the type $P\times G$ is treated analogously.

For (2), we use \cite[Theorem 2.2, (3)]{codenottisantos}, which shows that $\wdt(P\oplus Q) = \wdt(P)$. In order to see that $P\oplus Q$ is reduced, we note that the vertices of $P\oplus Q$ are of the form $(v,0)$, or $(0,w)$, where $v\in P$ and $w\in Q$ are vertices of $P$ and $Q$, respectively. A width-realizing lattice direction $y\in \ZZ^k$, which is uniquely maximized at $v$ leads to a width-realizing direction $(y,0)\in\ZZ^{k+l}$, which is uniquely maximized by $(v,0)$. Thus, $P\oplus Q$ is reduced.

For (3), we first consider a vector $z\in\ZZ^{k+l}\times\{0\}$. By projecting onto $\RR^{k+l}\times\{0\}$, we see that $\wdt(P\star Q; z) = \wdt(P\oplus Q; z)\geq \wdt(P) $, with equality, if $z$ is the canonical embedding of one of the width directions of $P$ or $Q$. If $z\in\ZZ^{k+l+1}$ with $z_{k+l+1}\neq 0$, we have
\begin{equation}
\label{eq:join}
\wdt(P\star Q;z) \geq |z\cdot((0,h)-(0,0))|\geq h \geq \wdt(P).
\end{equation}
Thus, $\wdt(P\star Q) = \wdt(P)$, which is attained by the canonical embedding of the width directions of $P$ and $Q$. These vectors also testify that $P\star Q$ is reduced.
\end{proof}

Lemma \ref{lemma:constr} and its proof lead us to a first example of a reduced 3-polytope whose width directions do not span its ambient space.

\begin{example}
    \label{ex:join}
    The simplex $S = \conv\{ (-1,0,0)^T,(1,0,0)^T,(0,1,3)^T,(0,-1,3)^T\}$ is reduced since it arises as a join as in Lemma \ref{lemma:constr}, where $P = Q = [-1,1]$ and $h=3$. As we saw in the proof, the width directions of $S$, $\pm e_1$ and $\pm e_2$, are given by the canonical embedding of the width directions of $P$. Since $h>2$, the inequality in \eqref{eq:join} is strict, so we have no width direction outside $\spann\{e_1,e_2\}$.
\end{example}

The remainder of this subsection is devoted to a lifting construction which allows us to construct origin-symmetric complete (reduced) polytopes whose diameter (width) directions are not full-dimensional.
The basic idea of the construction is to lift the normal vectors of a polytope $P\subset\RR^d$ into $\RR^{d+1}$ to obtain a new polytope $\overline P \subset\RR^{d+1}$ such that $\overline P \cap e_{d+1}^\perp = P$ and then adjust the lattice $\Lambda\subset\RR^{d+1}$ afterwards such that the lifted polytope is indeed complete.

The following lemma makes this geometric idea precise.

\begin{lemma}
    \label{lemma:lifting}
    Let $P=\{x\in\RR^d \colon |x\cdot a_i|  \leq 1,~1\leq i \leq m\}$ be an origin-symmetric $d$-polytope. Suppose that $m>d$ and each $a_i$ contributes a pair of facets to $P$. Then there exist numbers $\eta_1,\dots,\eta_m\in\RR$ such that the polyhedron
    \[
       \overline{P}= \{y\in \RR^{d+1} \colon  |y\cdot  \overline{a_i}|\leq 1,~1\leq i \leq m\}, 
    \]
    where $\overline{a_i}=(a_i,\eta_i)^T\in\RR^{d+1}$, is a bounded origin-symmetric $(d+1)$-polytope with $2m$ facets and $\overline{P}\cap e_{d+1}^
    \perp =P$. 
\end{lemma}

\begin{proof}
    It is clear that $\overline{P}\cap e_{d+1}^\perp =P$ holds for any choice of $\eta_1,\dots,\eta_m\in\RR$.
    By polarity, it suffices to prove that we can choose $\eta_1,\dots,\eta_m$ such that $\overline{P}^\star = \conv \{\pm \overline{a_1},\dots,\pm\overline{a_m}\}$ contains the origin in its interior and has all of the vectors $\overline{a}_i$ as vertices. Again by polarity, we have that $0\in \inter P^\star = \conv \{\pm a_1,\dots,\pm a_m\}$. Similarly to the proof of Steinitz' theorem \cite[Theorem 1.3.10]{schneider}, we may consider an arbitrary line $\ell$ passing through the origin and apply Caratheodory's theorem in the two opposite facets of $P^\star$ intersecting $\ell$ in order to choose $d$ linearly independent vectors $a_{i_1},\dots,a_{i_d}$ of the vectors $a_i$ such that the origin is contained in the interior of $\conv\{\pm a_{i_1},\dots,\pm a_{i_d}\}$.
    Since $m>d$, we may assume that $a_m$ is not among the vectors $a_{i_1},\dots, a_{i_d}$. We set $\eta_1 =\cdots = \eta_{m-1} = 0$ and $\eta_m = 1$. This way, $\overline{P}^\star$ is a double pyramid over $\conv\{\pm a_1,\dots,\pm a_{m-1}\}$, which contains the origin in its interior and the vectors $\overline{a_i}$, $1\leq i\leq m$, as vertices.
\end{proof}

If we perform a lifting of the normal vectors of $P$ as in the above Lemma and extend a lattice $\Lambda\subset\RR^d$ to a $(d+1)$-dimensional lattice $\overline{\Lambda}\subset\RR^{d+1}$ by adding a vector ``sufficiently far away from $\spann\Lambda$'' we can lift a complete polytope $P\subset\RR^d$ to a complete polytope $\overline{P}\subset\RR^{d+1}$ (cf.\ Figure \ref{fig:lifting}). This is formalized in the following proposition. 

\begin{figure}
    \centering
    \includegraphics[width = .7\textwidth]{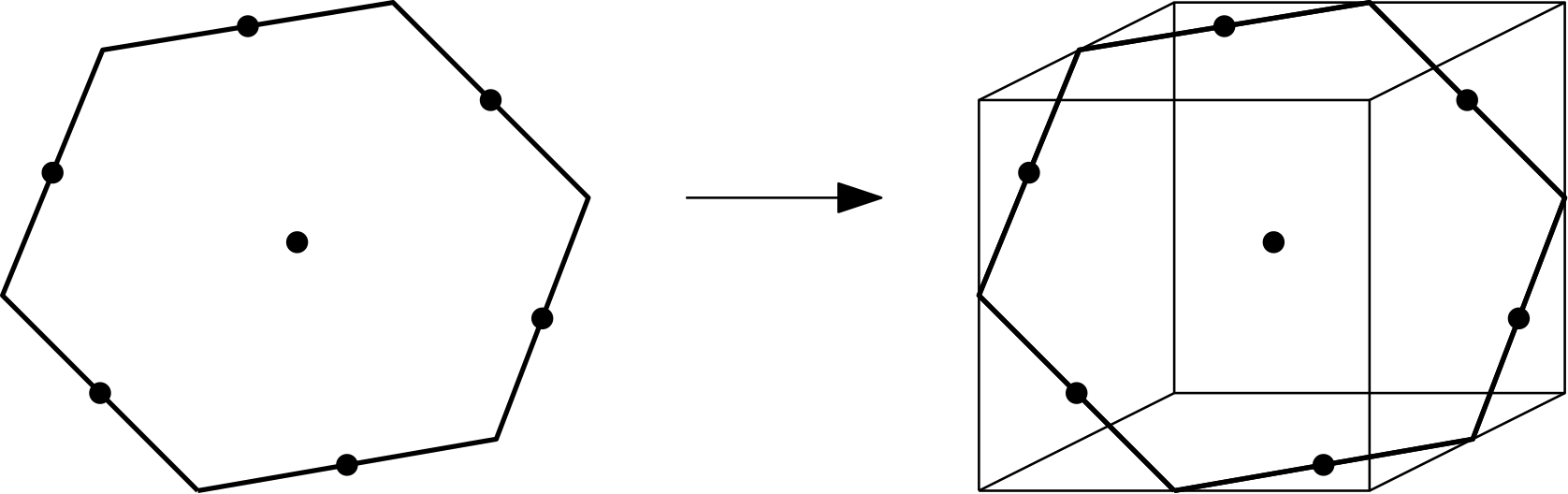}
    \caption{Lifting the normal vectors of the complete hexagon $P$ into $\RR^3$ yields a parallelotope $\overline{P}$ that is complete with respect to the 2-lattice $\Lambda$. The lattice $\Lambda$ can easily be extended to a 3-lattice $\overline{\Lambda}$ to which $\overline{P}$ is complete by choosing a plane parallel to $\Lambda$ which does not intersect $\overline{P}$.}
    \label{fig:lifting}
\end{figure}

\begin{proposition}
\label{prop:lifting}
Let $P\subset\RR^d$ be a complete origin-symmetric $d$-polytope with respect to some $d$-dimensional lattice $\Lambda$. If $P$ has more than $2d$ facets, there exists an origin-symmetric polytope $\overline{P}\subset\RR^{d+1}$ and a $(d+1)$-dimensional lattice $\overline{\Lambda}\subset\RR^{d+1}$ such that the following holds:
\begin{enumerate}
    \item $\overline{P}$ arises from $P$ via a lifting of the normals as in Lemma \ref{lemma:lifting}.
    \item $\overline{\Lambda}$ contains $\Lambda$ as a primitive sub-lattice.
    \item $\overline{P}$ is a complete polytope w.r.t.\ $\overline{\Lambda}$.
    \item The diameter directions of $\overline{P}$ are the same than the ones of $P$ (when the latter is embedded in $\RR^{d+1}$).
\end{enumerate}
\end{proposition}

\begin{proof}
    We apply Lemma \ref{lemma:lifting} to $P$ and obtain an origin-symmetric polytope $\overline{P}$ with the same number of facets and $\overline{P}\cap e_{d+1}^\perp = P$. We choose a vector $t\in\RR^{d+1}$ such that $\overline{P}\cap (t+\spann\Lambda) = \emptyset$. Then, $\overline{P}$ contains no lattice points of $\overline{\Lambda} = (\Lambda\times\{0\}) \oplus \ZZ t$ other than those in $P\cap \Lambda$. Since $\overline{P}$ is symmetric, its diameter is then attained by the same directions as for $P$ (cf.\ Lemma \ref{lemma:successive_minima}). It remains to check that $\overline{P}$ is complete. To this end, it is enough to note that for any boundary point $x\in\bd(P)$, we have $x\cdot (\pm a_i) = 1$ if and only if $(x,0)^T\cdot (\pm \overline{a_i})=1$. Hence, the segments that realized the diameter of $P$ intersect the facets of $\overline{P}$ in their relative interior and so $\overline{P}$ is complete.
\end{proof}

\subsection{Proofs of Theorems \ref{thm:polytopes} and \ref{thm:dimension}}

We want to use our construction methods in order to prove the main theorems \ref{thm:polytopes} and \ref{thm:dimension}.

\begin{proof}[Proof of Theorem \ref{thm:polytopes}]
        (1) is a direct consequence of Theorem \ref{thm:reduced_meta}, since $k\leq d$. (2) follows from Theorem \ref{thm:complete_meta} in the same way. In order to see that the number $2^{d+1}-2$ is best possible, we recall that the Voronoi cell of $A_d^\star$ is a permutohedron \cite[Ch.\ 4, Sec.\ 6.6]{conwaysloanebook}, so it has $2^{d+1}-2$ facets. Since $V_{A_d^\star}$ is complete with respect to $A_d^\star$ (cf.\ Proposition \ref{prop:voronoi}), the number $2^{d+1}-2$ is optimal with respect to $A_d^\star$. But in view of Lemma \ref{lemma:properties} (2), it is also optimal with respect to any other lattice. For reduced polytopes, the optimality is proven in the same way, since $(V_{A_d^\star})^\star$ is reduced with respect to $A_d$.
\end{proof}

\begin{proof}[Proof of Theorem \ref{thm:dimension}]
    (1) can be derived from Theorem \ref{thm:reduced_meta}, since a $d$-polytope has at least $d+1$ vertices. So it follows that $d+1 \leq 2^{k+1}-2$, where $k$ is the dimension of the linear space spanned by the diameter realizing segments. The claim follows by rearranging this inequality.

    (2) follows in the same fashion from Theorem \ref{thm:complete_meta}, since a $d$-polytope also has at least $d+1$ facets.

    In order to see that the logarithmic dependence on $d$ is optimal, we apply Proposition \ref{prop:lifting} to the permutohedron $P = V_{A_k^\star}$, where $k\geq 2$ is arbitrary, which is complete with respect to $\Lambda=A_k^\star$ (Here we think of $\spann(A_k^\star)$ as $\RR^k$). Since $P$ has $2^k-1>k$ pairs of opposite facets, we can apply Proposition \ref{prop:lifting} multiple times until we obtain an origin-symmetric polytope $Q$ of dimension $2^k-1$ which is complete with respect to some $(2^k-1)$-dimensional lattice $\Lambda\supset A_k^\star$ and whose diameter directions are exactly the ones of $P$. Thus, the diameter directions of $Q\subset\RR^{2^k-1}$ form a $k$-dimensional subspace, which proves the optimality of the logarithmic order in the complete case. For the reduced case it is again enough to use the symmetry of $Q$ in conjunction with Proposition \ref{prop:duality}.
\end{proof}

\subsection{Reductions and Completions}\label{ssec:reductions}

In the Euclidean case, it follows from Zorn's lemma that any convex body $C\subset \RR^d$ has a completion and a reduction. These are defined as complete (reduced) convex bodies that are supersets (subsets) of $C$ with the same Euclidean diameter (width) as $C$. 

In the discrete setting, the terms ``completion'' and ``reduction'' are defined in the same way with the appropriate definitions of diameter and width. We want to argue the existence of completions and reductions in a similar way to the Euclidean setting. However, the argument is more involved since the Euclidean diameter cannot be bounded by the lattice diameter. Therefore, it is not clear that the completion is indeed a bounded set. Likewise, when using Zorn's lemma to find a reduction, we have to make sure that it has non-empty interior to be a convex body in the sense of our definition. To this end we use the following two Lemmas.
\begin{lemma}
\label{lemma:vdc}
Let $K\subset\RR^d$ be an origin-symmetric $d$-dimensional convex set and let $\Lambda\subset\RR^d$ be a $d$-dimensional lattice such that $\inter K \cap \Lambda = \{ 0 \}$. Then, $K$ is bounded.
\end{lemma}

\begin{proof}
Suppose $K$ is unbounded. Then there exists a sufficiently large $R>0$ such that $\vol(K\cap [-R,R]^d) > 2^d\det\Lambda$. It follows from Minkowski's first theorem \cite[Theorem 2.5.1]{geometryofnumbers} that $K\cap [-R,R]^d$ contains a non-trivial interior lattice point of $\Lambda$, a contradiction.
\end{proof}

On the polar side, we have the following statement, where for an arbitrary bounded set $A$, the lattice width $\wdt_\Lambda(A)$ is defined similarly to
\eqref{eq:lattice_width}
as
\[
\wdt_\Lambda(A) =\inf_{y\in\Lambda^\star\setminus\{0\}}\sup_{a,b\in A} y\cdot(a-b).
\]

\begin{lemma}
    \label{lemma:vdc_polar}
    Let $K\subset\RR^d$ be a compact convex set and let $\Lambda\subset\RR^d$ be a $d$-dimensional lattice. If $\wdt_\Lambda(K) > 0$, then $K$ has non-empty interior.
\end{lemma}

\begin{proof}
    Let $w=\wdt_\Lambda(K)$ and consider the origin-symmetric convex set $M = w(K-K)^\star$. Then $M$ contains no non-trivial interior lattice points of $\Lambda^\star$: Suppose to the contrary that there exists $y\in \inter M \cap \Lambda^\star\setminus\{0\}$. Then we have $y\cdot (a-b) < w$ for all $a,b\in K$. But as $K$ is compact, there exist $a_0,b_0\in K$ such that
    \(
    \wdt(K;y) = y\cdot (a_0-b_0) < w,
    \)
    a contradiction.
    Moreover, since $K$ is bounded, $M$ is $d$-dimensional. Thus, Lemma \ref{lemma:vdc} applies and we obtain that $M$ is bounded. Consequently $K-K$, and thus also $K$, have non-empty interior.
\end{proof}

\begin{theorem}
\label{thm:zorn}
Let $C\subset\RR^d$ be a convex body and let $\Lambda\subset\RR^d$ be a $d$-lattice.
\begin{enumerate}
\item There exists a complete convex body $K\supseteq C$ with $\diam_\Lambda(K) = \diam_\Lambda(C)$.
\item There exists a reduced convex body $L\subseteq C$ with $\wdt_\Lambda(L)=\wdt_\Lambda(C)$.
\end{enumerate}
\end{theorem}

\begin{proof}
Without loss of generality, we assume that $0\in \inter C$. For (1), we consider the set \[
\mathcal{P} = \{K\supseteq C\text{ convex body}\colon \diam_\Lambda(K) = \diam_\Lambda(C)\},\]
partially ordered by inclusion. Since a completion of $C$ is a maximal element of $\mathcal{P}$, we aim to use Zorn's lemma in order to find such a maximal element. To this end, we consider a chain $\mathcal{C}\subseteq\mathcal{P}$ and we define $M^\prime = \bigcup \mathcal{C}$ as the union of all convex bodies in $\mathcal{C}$, as well as $M = \cl M^\prime$. If $M\in\mathcal{P}$, then Zorn's lemma applies and the proof is finished. It, thus, remains to show that $M$ is a convex body with $\diam_\Lambda(M) = \diam_\Lambda (C)$.

Clearly, $M$ has non-empty interior, since $C\subseteq M$. Since $\mathcal{C}$ is a chain, $M^\prime$ is convex, which implies that $M$ is convex as well. Moreover, the convexity of $M^\prime$ yields $\inter M \subseteq M^\prime$; Otherwise, there was a point $x\in \inter(M)\setminus M^\prime$, which could be separated from $M^\prime$ by a hyperplane and, since $x\in\inter(M)$, we may assume that the separation is strict. In this case, however, $x\not\in \cl M^\prime = M$, a contradiction. 

Now consider a primitive vector $v\in\Lambda\setminus\{0\}$ and let $I=[a,b]\subset M$ be a segment of maximal length in $M$ in direction $v$. Towards a contradiction, assume that $M$ is unbounded in direction $v$, i.e., $\diam(M;v)=\infty$. Since $0\in\inter C\subseteq\inter M$, it follows that $M\cap \spann\{v\}$ is unbounded. Since $\inter M \subset M^\prime$, we find that $M^\prime \cap\spann\{v\}$ is also unbounded. Hence, there exists a sequence of points $(a_n)_{n\in\NN}\subset M^\prime\cap\spann\{v\}$ with $\lim_{n\to\infty} |a_n|=\infty$. In particular, there exists $n\in\NN$ with $|a_n| > \diam_\Lambda(C)\cdot |v|$. Consider $A\in\mathcal{C}$ with $a_n\in A$.
We have
\[
\diam_\Lambda(A) \geq \Vol_1([0,a_n]) = \frac{|a_n|}{|v|} > \diam_\Lambda(C),
\]
a contradiction to $A\in\mathcal{P}$. Hence, the longest segment $I\subset M$ parallel to a given primitive vector $v\in\Lambda$ is of finite length. Moreover, if $I=[a,b]$, we have, since $0\in\inter M$, that
\[
I_n = \left(1-\frac 1n\right) I \subset \inter M\subset M^\prime.
\] 
It follows from the definition of $M^\prime$ that there exist convex bodies $A,B\in\mathcal{C}$ with $(1-1/n)a\in A$ and $(1-1/n)b\in B$. Since $\mathcal{C}$ is a chain, we can assume that $A\subseteq B$ and, thus, $I_n\subset B$. Hence, $\Vol_1(I_n) \leq \diam_\Lambda(B) = \diam_\Lambda (C)$ and, taking the limit $n\to\infty$, $\Vol_1(I) \leq \diam_\Lambda(C)$. Since $C\subseteq M$, we have $\diam_\Lambda(M)=\diam_\Lambda(C)$. It follows that the convex set 
\[
D = \frac{1}{\diam_\Lambda(C)}(M-M)
\]
contains no non-trivial interior lattice points of $\Lambda$. From Lemma \ref{lemma:vdc} we have that $D$ is bounded. This implies that $M$ is bounded as well.

In summary, we saw that $M\in\mathcal{P}$ is a valid upper bound of $\mathcal{C}$ and thus, by Zorn's lemma, the completion of $C$ exists.

For (2), we consider the set
\begin{equation}
    \label{eq:Q}
    \mathcal{Q} = \{ K\subseteq C\text{ convex body}\colon \wdt_\Lambda(K) = \wdt_\Lambda(C) \},
\end{equation}
partially ordered by inclusion. Again let $\mathcal{C}\subseteq \mathcal{Q}$ be a chain in $\mathcal{Q}$. As in the proof of (1) it suffices to show that the intersection $M = \bigcap \mathcal{C}$ of all convex bodies in $\mathcal{C}$ is an element of $\mathcal{Q}$. 

It is clear, that $M$ is a compact convex set. We have to show that $\inter M \neq \emptyset$ and $\wdt_\Lambda(M) = \wdt_\Lambda(C)$. To this end, we consider a primitive direction $y\in\Lambda^\star$ and, towards a contradiction, we assume that $\wdt(M;y) < \wdt_\Lambda(C)$. For a scalar $\alpha\in\RR$, let $H_\alpha = \{ x\in\RR^d \colon x\cdot y = \alpha\}$ and let $H_\alpha^-$ and $H_\alpha^+$ be the closed half-spaces defined by $H_\alpha$ depending on whether $x\cdot y$ is less than or greater than $\alpha$ respectively. Due to our assumption on $\wdt(M;y)$, there exist $\alpha<\beta$ with $M\subset H_\alpha^+ \cap H_\beta^-$ and $\beta-\alpha < \wdt_\Lambda(C)$.

Let $\alpha^\prime = \alpha-\varepsilon$ and $\beta^\prime = \beta+\varepsilon$, where $\varepsilon>0$ is such that $\beta^\prime - \alpha^\prime <\wdt_\Lambda(C)$. Suppose that each $A\in \mathcal{C}$ intersects $H_{\alpha^\prime}^-$. Then we consider the family of non-empty convex bodies 
\[
\mathcal{C}^\prime = \{A\cap H_{\alpha^\prime}^- \colon A\in\mathcal{C}\}.
\] 
Since $\mathcal{C}$ is a chain, so is $\mathcal{C}^\prime$. In particular, since all $A^\prime \in \mathcal{C^\prime}$ are non-empty, $\mathcal{C}^\prime$ satisfies the finite intersection property, i.e., for any finite collection $\{A_1^\prime,\dots, A_n^\prime\}\subseteq\mathcal{C}^\prime$, we have $\bigcap_{i=1}^n A_i^\prime\neq\emptyset$. Since any $A^\prime\in \mathcal{C}^\prime$ is contained in the compact set $C$, it follows that $M\cap H_{\alpha^\prime}^- = \bigcap \mathcal{C}^\prime\neq\emptyset$, which contradicts $M\subset H_\alpha^+$. Hence, there is a convex body $A\in \mathcal{C}$ with $A\subset H_{\alpha^\prime}^+$ and, likewise, a convex body $B\in\mathcal{C}$ with $B\subset H_{\beta^\prime}^-$. We may assume that $A\subset B$, because $\mathcal{C}$ is a chain. It follows that 
\[
\wdt(A;y) \leq \beta^\prime - \alpha^\prime < \wdt_\Lambda(C),
\]
a contradiction to $A\in\mathcal Q$. Thus, we have $\wdt(M;y)\geq \wdt_\Lambda(C)>0$, for all $y\in\Lambda^\star\setminus\{0\}$. Lemma \ref{lemma:vdc_polar} gives that $M$ has non-empty interior. Also, since $M\subseteq C$, it follows that $\wdt_\Lambda(M)=\wdt_\Lambda(C)$ as desired. So $M\in\mathcal{Q}$ is a minimal element of $\mathcal C$ and Zorn's lemma gives the existence of a minimal element $L\in\mathcal Q$, which is a reduction of $C$.
\end{proof}

\begin{remark}
\label{rem:zorn}
    \begin{enumerate}
        \item If the convex body $C$ in Theorem \ref{thm:zorn} is symmetric, then its reduction and completion may also be chosen to be symmetric. To see this, we start with the reduction of $C$. Replacing the partially ordered set $\mathcal Q$ in \eqref{eq:Q} by
        \[
            \mathcal{Q}_s = \{ K\subseteq C\text{ convex body}\colon \wdt_\Lambda(K) = \wdt_\Lambda(C),~K\text{ is origin-symmetric}\},
        \]
        Zorn's Lemma yields the existence of a minimal element $L\in\mathcal{Q}_s$. Indeed, $L$ is reduced. Otherwise, there would be an exposed point $p\in L$ at which the width is not uniquely attained (cf.\ Proposition \ref{prop:reduced_exposed}). The same is then true for $-p\in L$. Applying the truncation described in the proof of Proposition \ref{prop:reduced_exposed} to $p$ and $-p$ simultaneously yields an origin-symmetric convex body $L^\prime \subsetneq L$ with $\wdt_\Lambda(L^\prime)=\wdt_\Lambda(L)$, a contradiction to the minimality of $L$ in $\mathcal{Q}_s$. The existence of an origin-symmetric completion can now be deduced immediately from Theorem \ref{prop:duality}.
        
        \item The reduction and completion of $C$ are in general not unique. For instance, each $Q_x$, $x\in (0,1)$ in Example \ref{ex:dim2} constitutes a symmetric reduction of $[-1,1]^2$. Moreover, the right triangle $\conv\{\tbinom{1}{1}, \tbinom{1}{-1},\tbinom{-1}{-1}\}$ is a non-symmetric reduction.
        \item In \cite[Lemma 1.2]{coveringminima} it is shown that the product of the successive minima of $C$ and its polar is bounded from above by $c\cdot d$, where $c>0$ is an absolute constant, i.e., we have
        \begin{equation}
        \label{eq:wd_ineq}
        \frac{\wdt_\Lambda(C)}{\diam_\Lambda(C)} = \sukz_1(C-C;\Lambda)\,\sukz_1((C-C)^\star; \Lambda^\star)\leq c\cdot d,
        \end{equation}
        for any convex body $C\subset\RR^d$. Conway and Thompson showed that the linear order of this upper bound is achieved by certain ellipsoids \cite[Theorem 9.5]{bilinear}, i.e., there exists another universal constant $c^\prime$ such that for each $d\in\NN$ there exists an ellipsoid $\mathcal E\subset\RR^d$ such that
        \[
        \frac{\wdt_\Lambda(\mathcal E)}{\diam_\Lambda(\mathcal E)} \geq c^\prime\, d.
        \]
        Considering a reduction (completion) of $\mathcal{E}$ shows that the linear order in the inequality is also achieved by polytopes with at most $2^{d+1}-2$ vertices (facets).
    \end{enumerate}
\end{remark}

For polytopes we can construct a reduction explicitly without the help of Zorn's lemma. Moreover, this reduction does not increase the number of vertices of the polytope.

\begin{proposition}
    \label{prop:reduction}
    Let $P\subset\RR^d$ be a $d$-polytope and let $\Lambda\subset\RR^d$ be a $d$-lattice. Then there exists a reduced $d$-polytope $Q\subseteq P$ such that $\wdt_\Lambda(P)=\wdt_{\Lambda}(Q)$ and $|\ext(Q)|\leq |\ext(P)|$.
\end{proposition}

In particular, any simplex has a reduction which is also a simplex. The idea of the construction is to push the vertices at which the width is not uniquely achieved inside the polytope. This is formalized in the following lemma.

\begin{lemma}
    \label{lemma:pull}
    Let $P\subset\RR^d$ be a $d$-polytope and let $\Lambda\subset\RR^d$ be a $d$-lattice. Let $W\subset\Lambda^\star$ be the set of width directions of $P$ and let $v\in\ext(P)$ be a vertex at which $\wdt_\Lambda(P)$ is not uniquely attained, i.e., $\inter(\ncone(v;P))\cap W=\emptyset$. Then, for any $x_0\in\inter P$, there exists a number $\lambda_0\in [0,1)$ such that
    \begin{equation}
        \label{eq:plambda}
        P_\lambda = \conv\big((\ext(P)\setminus\{v\})\cup\{\lambda v + (1-\lambda)x_0\}\big)\subseteq P
    \end{equation}
    satisfies $\wdt_\Lambda(P_\lambda) = \wdt_\Lambda(P)$, for all $\lambda\in [\lambda_0,1]$. 
\end{lemma}

\begin{proof}
    We consider $\lambda_0 = \inf\{\lambda \in [0,1] \colon \wdt_\Lambda( P_\lambda) = \wdt_\Lambda(P)\}$. Since $\wdt_\Lambda$ is continuous, the infimum $\lambda_0$ is indeed attained. It remains to show that $\lambda_0 < 1$. To this end, we recall that for $\varepsilon>0$ sufficiently small, the width directions of $P_{1-\varepsilon}$ are given by $W$, the set of width directions of $P$ (cf.\ Lemma \ref{lemma:continuity}).
    Since $\inter(\ncone(v;P))\cap W=\emptyset$, it follows for any $y\in W$ that $\wdt(P;y)$ is attained by two vertices $a,b\in\ext(P)\setminus\{v\}$, which are still present in $P_{1-\varepsilon}$. We obtain that $\wdt(P_{1-\varepsilon}; y)=\wdt(P;y)$ and, thus, $\wdt_\Lambda(P_{1-\varepsilon}) = \wdt_\Lambda(P)$. This shows $\lambda_0<1$ and finishes the proof.
\end{proof}

\begin{proof}[Proof of Proposition \ref{prop:reduction}]
    Let $V=\ext (P)$ be the set of vertices of $P$. We call a vertex $v\in V$ reduced if there exists a width realizing direction $y\in\Lambda^\star\setminus\{0\}$ with $v\cdot y > x\cdot y$, for all $x\in P\setminus\{v\}$. We can thus rephrase Proposition \ref{prop:reduced_exposed} as follows: $P$ is reduced if and only if all of its vertices are reduced. 

    So assume that there exists a vertex $v\in V$ which is not reduced. We distinguish two cases.\\
    \\
    \textbf{Case 1:} $\wdt_\Lambda(\conv(V\setminus\{v\})) = \wdt_\Lambda(P)$.\\
    We set $P^\prime = \conv( V\setminus \{v\})\subset P$  and obtain a polytope inside $P$ that has the same width as $P$ and one non-reduced vertex less than $P$. The polytope $P^\prime$ is $d$-dimensional in view of Lemma \ref{lemma:vdc_polar}.\\
    \\
    \textbf{Case 2:} $\wdt_\Lambda(\conv(V\setminus\{v\})) < \wdt_\Lambda(P)$.\\
     We fix some $x_0\in\relint\conv(V\setminus\{v\})$ and let $x_\lambda = \lambda v + (1-\lambda) x_0\in \inter P$, for $\lambda \in [0,1]$. In view of Lemma \ref{lemma:pull}, there exists a minimal $\lambda_0$ such that $\wdt_\Lambda(\conv(V\setminus\{v\}\cup \{x_{\lambda_0}\}) = \wdt_\Lambda(P)$. We let $v^\prime = x_{\lambda_0}$ and $P^\prime = \conv (V\setminus\{v\}\cup\{v^\prime\})\subset P$. By construction, we have $\wdt_\Lambda(P^\prime) = \wdt_\Lambda(P)$. We claim that $v^\prime$ is a reduced vertex of $P^\prime$.

     Clearly, $v^\prime$ is a vertex of $P^\prime$ to begin with. Otherwise, it would not be necessary for the convex hull description of $P^\prime$ and we find ourselves in Case 1. For the same reason (and by the choice of $x_0$), we have $\lambda_0 > 0$. Let $W\subset\Lambda^\star$ be the set of width directions of $P^\prime$, i.e., $W = (\Lambda^\star\setminus\{0\})\cap \wdt_\Lambda(P)(P^\prime-P^\prime)^\star$. If $v^\prime$ is not a reduced vertex of $P^\prime$, then for each $y\in W$, there exist $u,w\in V\setminus\{v\}$ with $\wdt_\Lambda(P) = (u-w)\cdot y$. If we replace $v^\prime = x_\lambda$ by $x_{\lambda-\varepsilon}$, for $\varepsilon>0$ sufficiently small, we obtain a polytope $P^{\prime\prime}$ whose width directions are still among the vectors $W$ (cf.\ Lemma \ref{lemma:continuity}). But this implies $\wdt_\Lambda(P^{\prime\prime}) = \wdt_\Lambda(P)$, contradicting the minimality of $\lambda$. Hence, $v^\prime$ is indeed a reduced vertex of $P^\prime\subset P$. It is also clear by construction that $P^\prime$ is a $d$-polytope.\\
     \\
     In both cases, we constructed a $d$-polytope $P^\prime\subset P$ such that $\wdt_\Lambda(P^\prime) = \wdt_\Lambda(P)$ and $P^\prime$ has one non-reduced vertex less than $P$. Also, the total number of vertices of $P^\prime$ is not higher than the number of vertices in $P$. Thus, after finitely many steps, we obtain a reduced polytope with the desired properties.
\end{proof}

    Similar to Remark \ref{rem:zorn} i), if $P$ is origin-symmetric, a slight adaption of the proof of Proposition \ref{prop:reduction} yields an origin-symmetric reduction of $P$; we simply have to treat the opposite vertices of $P$ simultaneously. Using the polarity in the symmetric case, we can find a completion of $P$, which does not rely on Zorn's Lemma and does not possess more facets than $P$. It is obtained by dualizing the construction of the reduction of $P$, i.e., pairs of facets of $P$ are pushed outward until they either vanish or contain a lattice point in their interiors.

    However, in the non-symmetric case this process will not yield a completion. In fact, if $P$ is a simplex, it produces merely homothetic copies of $P$ that have a larger width. Therefore, the following problem remains open.
    
\begin{question}
    Is there an explicit way of constructing a completion of a (non-symmetric) polytope $P\subset\RR^d$ with respect to a given lattice $\Lambda\subset\RR^d$?
\end{question}

\section{Simplices and Triangles}
\label{sec:simplex}

\subsection{Simplices}
We begin with the proof of Theorem \ref{thm:local}.

\begin{proof}[Proof of Theorem \ref{thm:local}]
    Towards a contradiction, we assume that $S=\conv\{v_0,\dots,v_d\}\subset\RR^d$ is a local maximum of $\wdt_\Lambda$ on the class of hollow $d$-simplices, which is not reduced. Our goal is to find for any $\varepsilon>0$ a hollow simplex $S^\prime\subset\RR^d$ with $\hausdorff(S,S^\prime) <\varepsilon$ and $\wdt_\Lambda(S)<\wdt_\Lambda(S^\prime)$.

    If $S$ is not reduced, then there exists a vertex $v_0$, say, of $S$ such that $\inter(\ncone(v_0;S))\cap W=\emptyset$, where $W$ denotes the set of width directions of $S$. Using Lemma \ref{lemma:pull}, we can replace this vertex by a new vertex $v_0^\prime\in \inter S$ and obtain a simplex $T = \conv\{v_0^\prime,v_1,\dots,v_d\}\subset S$ with $\hausdorff(S,T) < \tfrac{\varepsilon}{2}$ and $\wdt_\Lambda(S)=\wdt_\Lambda(T)$. We consider an inequality description
    \begin{equation}
        \label{eq:ineqt}
        T = \{ x\in\RR^d \colon a_i\cdot x\leq b_i,~0\leq i\leq d\},
    \end{equation}
    for certain $a_i\in\RR^d\setminus\{0\}$ and $b_i\in\RR$. We assume that the new vertex $v_0^\prime$ is contained in the facets $F_i=\{x\in T\colon a_i\cdot x = b_i\}$, for $1\leq i\leq d$. Since $v_0^\prime\in\inter S$, we have $\relint F_d\subset \inter S$. Since $S$ is hollow, it follows that $T$ is hollow as well and $\relint (F_d) \cap \Lambda =\emptyset$. Therefore, there exists a number $\delta>0$ such that 
    \[
        S^\prime = \{x\in\RR^d\colon a_i\cdot x\leq b_i,~0\leq i < d,~a_d\cdot x\leq b_d+\delta\}
    \]
    is hollow and $\hausdorff(T,S^\prime)<\tfrac{\varepsilon}{2}$. So we have $\hausdorff(S,S^\prime)<\varepsilon$. Moreover, since $T$ is a simplex, $S^\prime$ is a translation of $\mu T$, for some scale factor $\mu>1$. This implies $\wdt_\Lambda(S^\prime) >\wdt_\Lambda(T)=\wdt_\Lambda(S)$. 
 \end{proof}

\begin{remark}
    \label{rem:locals}
    In the plane, Hurkens found a triangle $T_2\subset\RR^2$ that realizes $\flt(2)$ \cite{hurkens}. In higher dimensions, the realizers of $\flt(d)$ are unknown, but local maximizers of $\wdt_\Lambda$ within the class of hollow simplices have been found. In \cite{codenottisantos} a tetrahedron $T_3\subset\RR^3$ that locally maximizes the width has been obtained, and in \cite{weltgeetal} the authors construct a 4-dimensional local maximum $T_4\subset\RR^4$, as well as a 5-dimensional local maximum $T_5\subset\RR^5$. By Theorem \ref{thm:local}, these simplices are reduced. Moreover, a \texttt{sagemath} \cite{sagemath} computation shows that $T_3$ and $T_4$ are complete, while $T_5$ is not complete. The script that we used to check the reducedness and completeness of polytopes can be accessed under the following URL:
        \url{https://github.com/AnsgarFreyer/lattice_reduced_code}
\end{remark}

It remains open whether Theorem \ref{thm:local} holds true for arbitrary convex bodies. 
%We call a convex body $C\subset\RR^d$ a \emph{realizer} of $\flt(d)$ if $C$ is hollow and $\wdt_\Lambda(C) = \flt(d)$. 

 \begin{question}
 \label{quest:reduced_realizers}
     Is any realizer of $\flt(d)$ reduced?
 \end{question}
 
% Note that Question \ref{quest:reduced_realizers} is equivalent to the question whether any realizer $C$ of $\flt(d)$ is \emph{inclusion-maximal hollow}, i.e., $C$ is hollow, while any $C^\prime\supsetneq C$ is not hollow. Due to a result Lov\'asz, inclusion-maximal hollow polytopes have at most $2^d$ facets \cite{lovasz_2d}, see also \cite{Averkov2013}. 

We finish this subsection with an interesting observation concerning complete simplices.

\begin{proposition}
    \label{prop:complete_simplex}
    Let $S\subset\RR^d$ be a complete $d$-simplex with respect to some $d$-lattice $\Lambda$. Then, the diameter directions of $S$ span $\RR^d$.
\end{proposition}

\begin{proof}
    According to Proposition \ref{prop:endpoints}, for every facet $F\subset S$, there exists a diameter segment $I_F\subset S$ with an endpoint $x_F\in \relint F$. By Lemma \ref{lemma:2}, the other endpoint of $I_F$ lies inside the face opposite of $F$. Since $S$ is a simplex, this means that the other endpoint is the unique vertex $v_F\in S$ opposite to $F$.

    Towards a contradiction, assume that the directions $x_F-v_F$, where $F$ ranges over all facets of $S$, are contained in a $(d-1)$-dimensional space $H\subset\RR^d$. Let $t\in\RR^d$ such that $H+t$ supports $S$. Then, $S\cap (H+t)$ contains a vertex $v$ of $S$. Let $F$ be the facet opposite to $v$. Then, $\relint F$ is disjoint from $H+t$; otherwise $F\subset H+t$ and the entire simplex would be lower-dimensional. But now the segment $I_F$ is not parallel to $H$, a contradiction.
\end{proof}

\begin{remark}
    \label{rem:realizers}
     As we saw in Example \ref{ex:join}, there exist reduced simplices $S\subset\RR^d$ whose width directions $W$ do not span $\RR^d$. Consider any point of $x_0\in\inter S$. Since the polar $(S-x_0)^\star$ is again a simplex, the above proposition shows that $(S-x_0)^\star$ cannot be complete with the same diameter directions $W$, as it is the case for origin-symmetric convex bodies when polarized with respect to the origin (cf. Theorem \ref{prop:duality}).
\end{remark}

It would be interesting to find more classes of complete (or reduced) convex bodies for which the diameter (width) directions span the entire ambient space. One such class apart from simplices are the Voronoi cells from Proposition \ref{prop:voronoi}, since any Voronoi relevant vector of $\Lambda$ is a diameter direction. By duality the width directions of $(V_\Lambda)^\star$ are full-dimensional. 

\subsection{Classification of Triangles in the Plane} \label{classification_triangles}
In this subsection we give a complete classification of reduced triangles and triangles that are simultaneously reduced and complete. Finally, we show that every complete triangle is among the reduced and complete triangles, proving Theorem \ref{thm:triangles}. 

For convenience, in this subsection we only consider the lattice $\Lambda=\ZZ^2$ and we write $\diam(C)$ for $\diam_{\ZZ^2}(C)$, as well as $\wdt(C)$ for $\wdt_{\ZZ^2}(C)$. 

\begin{definition}
\label{def:equiv}
Two polygons $P,Q\subset \RR^2$ are called \emph{equivalent}, if there exist $\lambda>0$, $t\in\RR^2$ and $U\in\GL_2(\ZZ)$ such that $P=t+\lambda UQ$.
\end{definition}

If $P\subset\RR^2$ is a reduced polygon, then the primitive vectors in $\ZZ^2$ that realize $\wdt(P)$ must be 2-dimensional. Moreover, they contain a basis of $\ZZ^2$ (cf.\ \cite[Ch.1, Theorems 3.4 and 3.7]{geometryofnumbers}). Thus, we have the following lemma.

\begin{lemma}
    \label{lemma:square}
    Let $P\subset\RR^2$ be reduced w.r.t.\ $\ZZ^2$. Then $P$ is equivalent to a polygon $Q\subset[-1,1]^2$, such that each edge of $[-1,1]^2$ contains a vertex of $Q$ and $\wdt(Q)=2$.
\end{lemma}

First, consider a reduced triangle $T\subset[-1,1]^2$ with $\wdt(T)=2$ such that every edge of $[-1,1]^2$ contains a vertex of $T$. Clearly, one of the vertices of $T$ is then a vertex of $[-1,1]^2$ as well. Without loss of generality, let $(-1,-1)^T$ be a vertex of $T$. The other two vertices are then given by $(x,1)^T$ and $(1,y)^T$ for certain $x,y\in [-1,1]$. 

Since $T$ is reduced, we must have $x+y\leq 0$; otherwise the vector $(-1,1)^T$ would achieve a smaller width. After reflecting at the diagonal if necessary, we can also assume that $x\geq y$. Indeed, these conditions are also sufficient for $T$ to be reduced:

\begin{proposition}
\label{prop:reduced_triangles}
A triangle $T\subset\RR^2$ is reduced with respect to $\ZZ^2$ if and only if it is equivalent to a triangle of the form
\begin{equation}
    \label{eq:txy}
    T_{xy} = \conv\left\{ \binom{-1}{-1}, \binom{x}{1}, \binom{1}{y}\right\},
\end{equation}
where $x+y\leq 0$ and $-1\leq y\leq x\leq 1$.
\end{proposition}

This result was proven independently by Cools and Lemmens in \cite{coolslemmens} for lattice triangles.

\begin{proof}
    It only remains to prove the sufficiency, i.e., we aim to prove that $T_{xy}$ is reduced for any $x,y\in [-1,1]$ with $x+y\leq 0$. Since $\wdt(T_{xy};e_i)=2$, $i=1,2$, we have $\wdt(T_{xy})\leq 2$ and if equality holds, then $T_{xy}$ is reduced by Proposition \ref{prop:reduced_exposed}.

    So we consider a vector $v=(a,b)^T\in\ZZ^2\setminus\{0\}$ with $a\neq 0$ and $b\neq 0$ and we distinguish two cases (note that the width does not depend on the orientation of $v$).\\
    \\
    \textbf{Case 1:} $a,b>0$.\\
    We consider the vertices $(x,1)^T$ and $(-1,-1)^T$. We have
    \[
    \wdt(T_{xy};v)\geq \left| \binom{a}{b}\cdot\left( \binom{x}{1} - \binom{-1}{-1} \right) \right| = a(x+1)+2b \geq 2b \geq 2,
    \]
    where we used that $x+1\geq 0$ and that $b$ is an integer.\\
    \\
    \textbf{Case 2:} $a>0>b$.\\
    We consider the vertices $(x,1)^T$ and $(1,y)^T$. We have
    \[\begin{split}
    \wdt(T_{xy};v)  \geq \left| \binom{a}{b} \cdot \left( \binom{1}{y} - \binom{x}{1}\right)\right|& = a(1-x) + b(y-1) \geq 1-x - (y-1)\\
    &=2-(x+y)\geq 2
    \end{split}\]
    where we used that $1-x\geq 0$, $y-1\leq 0$ and $a,b\in \ZZ$ in the first line, as well as $x+y\leq 0$ in the second line. 
\end{proof}

\begin{remark}
\label{rem:hurkens}

    Proposition \ref{prop:reduced_triangles} can be used to reprove Hurkens' result from \cite{hurkens} stating that $\flt(2)=1+\tfrac{2}{\sqrt 3}$. We sketch the argument here: It is enough to study the width of inclusion-maximal hollow convex bodies. Recall that $C$ is ``inclusion-maximal hollow'' if it contains no interior lattice points, but any convex body $C^\prime\supsetneq C$ contains an interior lattice point. Such bodies in the plane are polygons which contain a lattice point in the relative interior of each edge and thus, by a parity argument, they are either triangles or quadrangles (see \cite{lovasz_2d} or \cite{Averkov2013} for details). A short and elementary calculation shows that the maximum lattice width of an inclusion-maximal hollow quadrangle is 2 \cite{hurkens}. So the crux of Hurkens' proof is to show that $\flt_s(2)=1+\tfrac {2} {\sqrt{3}}$, where again $\flt_s(2)$ is the maximum lattice width of a hollow triangle.
    
    In view of Theorem \ref{thm:local}, we know that $\flt_s(2)$ is attained by a triangle $T$ which is both inclusion-maximal hollow \emph{and} reduced. Hence, $T$ is equivalent to a triangle of the form $T_{xy}$ of equation \eqref{eq:txy}, where $x+y\leq 0$ and $-1\leq y\leq x \leq 1$. Of course, our notion of equivalence does not preserve hollowness, since translations and scalings do not. Therefore, we are interested in the minimal $\mu > 0$ for which there exists a vector $v\in\RR^2$ and real numbers $-1\leq x\leq y\leq 1$ such that $x+y\leq 0$ and $T_{xy}$ is hollow with respect to the affine square lattice $\mathcal A = \mu\ZZ^2+v$. Since the lattice width of $T_{xy}$ is always 2, we then have $\flt_s(2) = 2/\mu$.

    Let $(x,y)$ be in the domain $D=\{(x,y)\colon -1\leq x\leq y\leq 1,~x+y\leq 0\}$, and suppose that $T_{xy}$ is inclusion-maximal hollow with respect to an affine square lattice $\mathcal A = \alpha\ZZ^2 + v$. Then, every edge of $T_{xy}$ contains a point of $\mathcal A$ in its relative interior. The coordinates of these points are either equal or they differ by at least $\alpha$. From this, one can derive that the points form an isosceles right triangle $R$ with edge direction $e_1$, $e_2$ and $\eins$ and the lengths of the catheters is $\alpha$. Thus, there are only two choices for $\alpha$ so that $T_{xy}$ can be inclusion-maximal hollow with respect to $\mathcal A$, depending on whether the right angle of $R$ is on the edge $E_1=[(-1,-1)^T,(x,1)^T]$ or $E_2=[(-1,-1)^T, (1,y)^T]$. Since $x\geq y$, choosing $R$ so that the right angle is on $E_1$ will yield the smaller $\alpha$. We denote this number by $\alpha(x,y)$ (cf.\ Figure \ref{fig:hurkens}). By computing the coordinates of the unique point $p\in E_1$  whose vertical distance to $E_2$ is the same as its horizontal distance to $[(1,y)^T, (x,1)^T]$, one can express $\alpha(x,y)$ explicitely as a function in $x$ and $y$ with domain $D$. The minimum of this function can be computed in exact arithmetic, which then confirms $\flt_s(2)=2/\min_{x,y}\alpha = 1 + \tfrac{2}{\sqrt{3}}$. The computations were carried out using \texttt{sagemath} \cite{sagemath} and the code can be found in \url{https://github.com/AnsgarFreyer/lattice_reduced_code}.
    \hfill $\diamond$
\end{remark}

\begin{figure}[!htb]
        \centering
        \includegraphics[width= .4\textwidth]{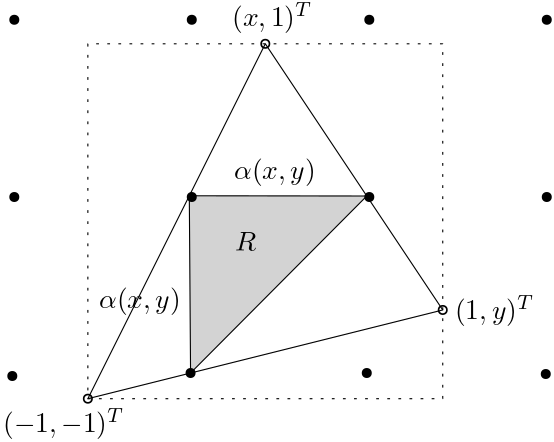}
        \caption{The situation in Remark \ref{rem:hurkens}.}
        \label{fig:hurkens}
\end{figure}

Now we turn our attention to the case when the triangle $T_{xy}$ as in equation \eqref{eq:txy} is also complete. From Proposition \ref{prop:complete_simplex}, we see that each vertex of $T_{xy}$ needs to be the endpoint of a diameter segment. The diameter directions of $T_{xy}$ are the lattice points in $\sukz_1(C;\ZZ^2)C$, where $C=T_{xy}-T_{xy}\subset[-2,2]^2$. Since the origin is an interior point of $T_{xy}$ (unless $x=y=-1$ in which case $T_{xy}$ is not complete), the diameter of $T_{xy}$ is greater than 1. Thus, it is realized by lattice points in $\inter C \cap [-2,2]^2 \cap\ZZ^2\subseteq [-1,1]^2$.

So the only possible direction of a segment $I_1$ that connects $(-1,-1)^T$ with its opposite edge and that realizes the diameter of $T_{xy}$ is $(1,1)^T$. A routine computation shows that the lattice length of $I_1$ is 
\[
\Vol(I_1)= \frac{3-xy-(x+y)}{2-(x+y)}.
\]
At the vertex $(x,1)^T$, the diameter could be obtained in the directions $(0,-1)^T$ and $(1,-1)^T$. But $(0,-1)^T$ clearly gives the longer segment $I_2$ and is therefore diameter realizing. The length of $I_2$ is given by
\[
\Vol_1(I_2)= \frac{3-y}{2}-\frac{1+y}{2}x.
\]
Note that, by exchanging the roles of $x$ and $y$ the third segment $I_3$ in direction $(-1,0)^T$ that connects the vertex $(1,y)^T$ to its opposite edge is of length 
\[
\Vol_1(I_3) = \frac{3-x}{2}-\frac{1+x}{2}y = \frac{3-y}{2}-\frac{1+y}{2}x =\Vol_1(I_2).
\]
Combining these equations we obtain, that $T_{xy}$ is complete, only if the following equation is fulfilled:
\[
 \frac{3-xy-(x+y)}{2-(x+y)} = \frac{3-y}{2}-\frac{1+y}{2}x
\]
After multiplying with $2-(x+y)$, we obtain an equation which is quadratic in $x$. Its solutions are $x=\frac{3-y}{y+1}$ and $x=-y$. The first solution is not in the interval $[-1,1]$ for any value of $y\in (-1,1)$. So we obtain that $T_{xy}$ is complete, only if $x=-y$. On the other hand, for $|x|<1$ and $x=-y$, the three segments $I_1$, $I_2$ and $I_3$ pass from their respective vertices to the interior of the opposite edge of that vertex and are all of the same length. The only direction for a longer rational segment would be $(1,-1)^T$, but we already saw that a segment in this direction has shorter lattice length than $I_2$ (and $I_3$). Since for $|x|= 1$, the triangle $T_{-xx}$ is not complete, we can now classify the triangles that are both complete and reduced as follows:

\begin{proposition}
\label{prop:triangle_classification}
A triangle $T\subset\RR^2$ is simultaneously reduced and complete with respect to $\ZZ^2$ if and only if it is equivalent to a triangle of the form
\[
T_x = \conv\left\{\binom{-1}{-1},\binom{x}{1}, \binom{1}{-x}\right\},
\]
where $0\leq  x < 1$.
\end{proposition}

Next, we show that the triangles obtained in Proposition \ref{prop:triangle_classification} are indeed the only complete triangles up to equivalence. For this, we need two elementary lemmas.

\begin{lemma}
\label{lemma:diagonal_point}
    Let $C\subset\RR^2$ be an origin-symmetric convex body such that $\inter C\cap \ZZ^2=\{0\}$ and $\{\pm e_1, \pm e_2\}\subset C$. If $C$ contains a lattice point other than $0$, $\pm e_1$, or $\pm e_2$, it also contains one of the points $(1,1)^T$ and $(1,-1)^T$.
\end{lemma}

\begin{proof}
    Let $x=(a,b)\neq \pm e_i$ be a non-zero lattice point in $C$. As the standard basis vectors are on the boundary of $C$, $x$ must be in the interior of one of the four quadrants. Without loss of generality, we can assume that $x\in\ZZ_{>0}^2$. Also, we can assume that $a\geq b$. We may write
    \[
    \binom{1}{1} = \frac{1}{a} x + \left(1-\frac ba\right) e_2.
    \]
    Since $a$ and $b$ are integers with $0<b\leq a$, it follows that $1-b/a = (a-b)/a\in [0,1]$ and $1/a + (1-b/a) = (1-b+a)/a \leq 1$. Hence, $(1,1)^T\in \conv\{0,e_2,x\}\subset C$.
\end{proof}

\begin{lemma}
\label{lemma:disjoint_segments}
    Consider two disjoint segments $I_i=[v_i,w_i]$. $i=1,2$, and define the point $v_3$ as the intersection of the lines $\aff\{v_1,w_2\}$ and $\aff\{v_2,w_1\}$ (if any). Then the triangle $T=\conv\{v_1,v_2,v_3\}$ does not contain both of the segments $I_1$ and $I_2$. 
\end{lemma}

\begin{proof}
    Let $\ell_1=\aff\{v_1,w_2\}$ and $\ell_2 = \aff\{v_2,w_1\}$ and suppose that $I_1\subset T$, i.e., $T$ contains the triangle $T^\prime=\conv\{v_1,v_2,w_1\}$. Since $v_1$ and $w_1$ are on the two edges of $T$ adjacent to $v_2$, any segment $I\subset T$ that connects $v_2$ to the line $\ell_1$ must intersect $[v_1,w_1]=I_1$. But $I_1$ and $I_2$ are disjoint, so $I_2$ cannot be contained in $T$.
\end{proof}

\begin{proof}[Proof of Theorem \ref{thm:triangles}]
    Let $T$ be a complete triangle. After applying a unimodular transformation and a dilation we may assume that $\diam(T)=1$ and that the standard basis vectors are among the diameter realizing directions of $T$. Thus, there is a horizontal segment $I_1\subset T$ and a vertical segment $I_2\subset T$ that both realize the diameter of $T$. By Proposition \ref{prop:complete_simplex}, $I_i$ connects a vertex $v_i$ to a point $w_i$ in the relative interior of the edge $E_i$ opposite to $v_i$, $i=1,2$. Again, after applying a unimodular transformation if necessary, we can assume that the right endpoint of $I_1$ is $v_1$ and that the upper endpoint of $I_2$ is $v_2$. The triangle $T$ is now uniquely described; the third vertex $v_3$ has to be the intersection of the lines $\aff\{v_1,w_2\}$ and $\aff\{v_2,w_1\}$ (in particular, these lines are not parallel).

    In view of Lemma \ref{lemma:disjoint_segments}, $I_1$ and $I_2$ intersect. After a translation, we may assume that $I_1=[(x-1,y)^T,(x,y)^T]$ and $I_2=[0,e_2]$, for certain $0<x,y< 1$ (Note that we cannot have equality in one of the constraints on $x$ and $y$, since then, one of the segments $I_i$ would be an edge of $T$). In order for the other endpoints $0$ and $(x-1,y)^T$ to be part of the triangle $\conv\{v_1,v_2,v_3\}$ it is necessary that $x> y$.

    By Proposition \ref{prop:complete_simplex}, there exists a third pair of lattice directions in which the lattice diameter of $T$ is attained. Since by assumption, $\diam(T)=1$, the difference body $C=T-T$ fulfills the assumptions of Lemma \ref{lemma:diagonal_point}. So the diameter of $T$ has to be attained in direction $(1,-1)^T$ or $(1,1)^T$ as well. The direction $(1,-1)^T$ can be ruled out; If it was diameter realizing, the corresponding segment in $T$ had to include $v_1$ as a vertex. But then the segment would have a shorter lattice length than $I_1$. So there is a diameter realizing segment $I_3$ in direction $(1,1)^T$. The only vertex that this segment can contain is $v_3$. So we have $I_3 = [v_3,w_3]$, where $w_3$ is the intersection of $v_3+\spann\{(1,1)^T\}$ with the edge $E_3=[v_1,v_2]$.

    We saw that $T$ is complete, if and only if it is equivalent to a triangle of the form $\conv\{e_2,(x,y)^T,v_3\}$, where $0<y<x<1$ and $\Vol_1(I_3)=1$ (with $v_3$ and $I_3$ as above). In the the next step we fix the value of $x\in (0,1)$ and show that there is exactly one value of $y\in (0,x) $ such that $\Vol_1(I_3)=1$. In this setting the points $v_3$ and $w_3$ depend on $y$, which is why we denote them by $v_3(y)$ and $w_3(y)$.
    Computing the point of intersection of the lines $\aff\{v_1,w_2\} = \aff\{(x,y)^T,(0,0)^T\}$ and $\aff\{v_2,w_1\} = \aff\{(0,1)^T,(x-1,y)^T\}$ shows that 
    \[
    v_3(y) = \begin{pmatrix}
        \frac{x(x-1)}{x-y}\\
        \frac{y(x-1)}{x-y}
    \end{pmatrix}.
    \]
   In particular, as $y$ increases, $v_3(y)$ moves strictly monotonously downward on the diagonal line $\ell=\{ (a,b)^T : a-b = x-1\}$. The degenerate cases $y\to 0$ and $y\to x$ correspond to $(x-1,0)$ and the ``point at infinity'' of $\ell$. Since $w_3(y)$ is by definition the intersection of $v_3(y)+\spann\{(1,1)^T\}$ with $[e_1,(x,y)^T]$, it follows from the fact that all $v_3(y)$ lie on the same diagonal $\ell$, that $w_3(y)$ is the intersection of $\ell$ and $[e_1,(x,y)^T]$. Since $\ell$ is independent of $y$, the point $w_3(y)$ moves strictly monotonously upward on $\ell$ as $y$ increases. In total, we find that $\Vol_1(I_3(y))$ is a strictly monotonous function in $y\in (0,x)$ (see also Figure \ref{fig:v3y}). By considering, e.g., a projection on a horizontal line, we see that $\lim_{y\to 0} \Vol_1(I_3(y)) < 1$. On the other hand, since $v_3(y)$ is unbounded for $y\to x$, we have $\lim_{y\to x} \Vol_1(I_3(y)) = \infty$. Consequently, for any $x\in (0,1)$, there is exactly one $y_x\in (0,x)$ such that the triangle $\Delta_x = \conv\{e_2,(x,y_x)^T,v_3(y_x)\}$ is complete. 
    
    To conclude, we recall the triangles $T_x = \conv\{(-1,-1)^T,(x,1),(1,-x)^T\}$ from Proposition \ref{prop:triangle_classification}. Let $r\in (0,1)$ be the fraction of the horizontal diameter realizing segment that is to the right of the vertical diameter segment. By considering the limits $x\to \pm 1$, we see that $r$ can attain any value between 0 and 1. Thus, the triangle $\Delta_r$ is in fact equivalent to one of the triangles $T_x$, which finishes the proof.
        \begin{figure}[!htb]
        \centering
        \includegraphics[width= .4\textwidth]{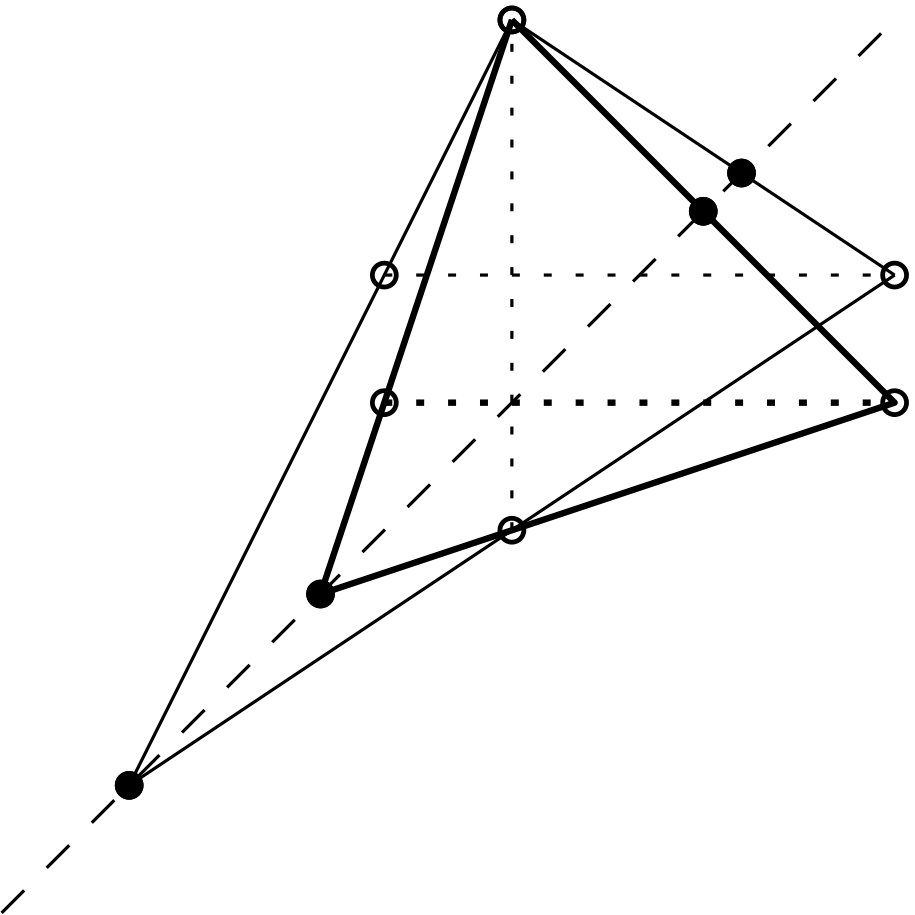}
        \caption{Illustration of the monotonicity of $y\mapsto \Vol_1(I_3(y))$. As $y$ increases, the black vertex $v_3(y)$ moves downward on the diagonal dashed line $\ell$, while the other endpoint $w_3(y)$ of $I_3(y)$ moves upward on $\ell$.}
        \label{fig:v3y}
    \end{figure}
\end{proof}

At this point, it is natural to ask whether Theorem \ref{thm:triangles} extends to arbitrary $d$-simplices. The following construction, however, leads to counterexamples for $d=3$ (Remark \ref{rem:3simplex} (3)).

\subsection{A Family of Complete Tetrahedra}

We start with the regular tetrahedron
\begin{equation}
\label{eq:reg_simplex}
\Delta = \conv\{0, e_i+e_j \colon 1\leq i < j\leq 3\}\subset [0,1]^3
\end{equation}
and we let \[
D=\Delta - \Delta = \conv\{ \pm e_i \pm e_j\colon 1\leq i < j\leq 3\}\subset C_3
\]
be its difference body. Note that $D$ is the convex hull of the midpoints of the edges of the cube $C_3$. Thus, $D$ has 6 quadrilateral facets corresponding to the original facets of $C_3$ and 8 triangular facets corresponding to the 8 vertices of $C_3$. Moreover, each quadrilateral facet is of the form $e-f$, where $e,f\subset\Delta$ are opposite edges of $\Delta$, and each triangular facet is of the form $\pm (F-v)$, where $F\subset\Delta$ is a facet of $\Delta$ and $v\in \Delta$ is the vertex opposite to it.

Note that a segment $I = [a,b]\subset\Delta$ passes from a vertex $v\in \Delta$ to the interior of the facet opposite to $v$, if and only if, $b-a$ is in the interior of one of the triangular facets of $D$. Hence, $\Delta$ is complete with respect to a 3-lattice $\Lambda\subset\RR^3$, if and only if  for some $\lambda>0$, we have $\inter (\lambda D) \cap \Lambda =\{0\}$ and $\lambda D$ possesses a lattice point of $\Lambda$ in the relative interior of each of its triangular facets.

For each choice of signs $\sigma\in \{\pm 1\}^3$, there is a triangular facet
\[
F_\sigma = \{ x\in C_3 \colon \sigma \cdot x = 2\}\subset D
\]
of $D$. We start with a point $(\alpha,\beta,\gamma)^T\in \inter F_{(1,1,1)}$, i.e., we have $\alpha+\beta+\gamma = 2$ and $\alpha,\beta,\gamma < 1$. Moreover, we want to require that $\alpha,\beta,\gamma\geq \frac 12$. Now let 
\begin{equation}
\label{eq:sym_lattice}
\Lambda = \spann_{\ZZ}\left\{
\begin{pmatrix}
-\alpha\\
\beta\\
\gamma
\end{pmatrix},
\begin{pmatrix}
    \alpha\\
    -\beta\\
    \gamma
\end{pmatrix},
\begin{pmatrix}
    \alpha\\
    \beta\\
    -\gamma
\end{pmatrix}
\right\}.
\end{equation}
Since
\[
\begin{pmatrix}
    \alpha\\
    \beta\\
    \gamma
\end{pmatrix}
=
\begin{pmatrix}
-\alpha\\
\beta\\
\gamma
\end{pmatrix}+
\begin{pmatrix}
    \alpha\\
    -\beta\\
    \gamma
\end{pmatrix}+
\begin{pmatrix}
    \alpha\\
    \beta\\
    -\gamma
\end{pmatrix},
\]
the difference body $D$ contains a lattice point of $\Lambda$ in each of the triangular facets $F_\sigma$, $\sigma\in\{\pm 1\}^3$. It remains to check that $\inter D \cap \Lambda = \{0\}$. To this end, we first consider a point of the form $x=(0,b,c)\in \inter D \cap \Delta$. There exist $z_1,z_2,z_3\in\ZZ$ with 
\[
\begin{pmatrix}
    0\\
    b\\
    c
\end{pmatrix}
=
z_1\begin{pmatrix}
    -\alpha\\
    \beta\\
    \gamma
\end{pmatrix}
+z_2\begin{pmatrix}
    \alpha\\
    -\beta\\
    \gamma
\end{pmatrix}
+z_3\begin{pmatrix}
    \alpha\\
    \beta\\
    -\gamma
\end{pmatrix}.
\]
It follows from the first equation that $z_1 = z_2+z_3$. Thus, the lattice point $x$ is of the form
\[
x=\begin{pmatrix}
    0\\
    2z_3\beta\\
    2z_2\gamma\\
\end{pmatrix}.
\]
If either $z_3$ or $z_2$ is non-zero, then, since $\beta,\gamma\geq \frac 12$, the corresponding entry has an absolute value of at least 1. But then $x$ is on the boundary of $C_3$. Hence, $z_2=z_3=0$, which implies $x=0$. 

By symmetry, we can conclude that a non-zero point in $x\in \inter D\cap\Lambda$ cannot have a zero coordinate. So, towards a contradiction, let $x=(a,b,c)^T\in \inter D\cap \Lambda$, where $a,b,c\in\RR\setminus\{0\}$.  Since both $D$ and $\Lambda$ are symmetric with respect to the coordinate hyperplanes, we may even assume that $a,b,c > 0$.  Since $x\in\Lambda$, there are positive integers $z_1,z_2,z_3\geq 1$ such that $x = (z_1\alpha, z_2\beta, z_3\gamma)^T$. It now follows from $x\in \inter D$ and the symmetry properties of $D$ that $(\alpha,\beta,\gamma)^T\in \inter D$, a contradiction to $(\alpha,\beta,\gamma)^T\in F_{(1,1,1)}\subset\bd D$.

So we have established that $\inter D\cap\Lambda = \{0\}$. 
This proves the following proposition.

\begin{proposition}
\label{prop:3simplex}
The regular simplex $\Delta$ defined in \eqref{eq:reg_simplex} is complete with respect to the lattice $\Lambda$ as defined in \eqref{eq:sym_lattice} for any choice of
\[
\begin{pmatrix}
    \alpha\\
    \beta\\
    \gamma
\end{pmatrix}
\in \left\{
\begin{pmatrix}
    a\\
    b\\
    c
\end{pmatrix}
\colon
\frac 12 \leq a,b,c < 1,~a+b+c = 2
\right\}.
\]
\end{proposition}

\begin{remark}
\label{rem:3simplex}
    \begin{enumerate}
        \item There is no restriction in considering the regular tetrahedron $\Delta$; Since all simplices are affinely equivalent, finding a lattice with respect to which a fixed simplex is complete is the same as finding a simplex which is complete with respect to a fixed lattice. 
        \item If one of the parameters ($\alpha$, say) is equal to $\frac 12$, the difference body $D$ contains the point 
        \[
        \begin{pmatrix}
            1\\
            0\\
            0
        \end{pmatrix}
        =\begin{pmatrix}
            \frac 12\\
            \beta\\
            \gamma
        \end{pmatrix}
        +\begin{pmatrix}
            \frac 12\\
            -\beta\\
            -\gamma
        \end{pmatrix} \in \Lambda,
        \]
        which is then a boundary point in one of the quadrilateral facets of $D$. For the tetrahedron $\Delta$, this implies that the diameter is not only attained by the four segments connecting the facets to their opposite vertices as guaranteed by Proposition \ref{prop:complete_simplex}, but also by an additional segment between two opposite edges.
        \item For a generic choice of the parameters, such as $(0.65, 0.65, 0.7)$, the tetrahedron $\Delta$ is not reduced with respect to $\Lambda$, as can be checked using \texttt{sagemath} (cf.\ Remark \ref{rem:locals}). This shows that Theorem \ref{thm:triangles} does not hold in dimension 3.
        \item The same tetrahedron has the property that also $(\Delta - \centroid(\Delta))^\star$ is not reduced with respect to $\Lambda^\star$. Thus, it confirms \eqref{eq:implication2} from Remark \ref{rem:non_symmetric_duality}.
    \end{enumerate}
\end{remark}

\section{Comparison to the Euclidean Case}
\label{sec:euclid}

 The study of Euclidean reduced convex bodies is in part motivated by the P\'{a}l-Kakeya problem, which asks for the maximal Euclidean width among convex bodies $C\subset\RR^d$ with $\vol(C)\leq 1$.
This problem is similar in spirit to the flatness problem, in which the Euclidean width is replaced by the lattice width and the volume constraint is replaced by the condition $|\inter C \cap \Lambda| < 1$.
 Since the volume functional is strictly monotonous, it is clear that all the bodies attaining the maximum width in the P\'al-Kakeya problem are Euclidean reduced. It was shown by P\'al \cite{pal} in 1921 that for $d=2$, the maximum is achieved by the regular triangle. In higher dimensions, however, the problem has been open for more than a century. The fact that the natural generalization of the regular triangle, the regular simplex, is not Euclidean reduced for $d\geq 3$ \cite{martini_simplices} shows that higher dimensions require a different approach. 

% In recent years there has been a lot of research on reduced and complete convex bodies also in non-Euclidean settings, such as spherical geometry and hyperbolic geometry (see, e.g., \cite{spherical_reduced, spherical_complete, sagmeister, leichtweiss, bezdek}), or Minkowski spaces \cite{minkowski_space, minkowski_space2, gonzalez_et_al}.

%In addition to the flatness problem,
A more homogeneous lattice variant of the P\'al-Kakeya problem, in which the study of lattice reduced convex bodies could play an important role, is the following. One replaces the Euclidean volume $\vol(K)$ by the \emph{lattice-normalized volume} $\Vol_\Lambda(K):=\tfrac{\vol(K)}{\det\Lambda}\leq 1$. Thus the task is to find the maximum lattice width $\wdt_\Lambda(K)$ among convex bodies $K$ with normalized volume $\Vol_\Lambda(K)\leq 1$. Makai conjectured that the maximum is $\sqrt[d]{\tfrac{2^dd!}{d+1}}$ and that it is attained only by simplices \cite{makai}. For instance, the simplex $S_d=\conv\{-\eins,e_1,\dots,e_d\}$ has width exactly the conjectured maximum, if it is dilated so that its volume is 1. We refer to \cite{gonzalez_schymura, henk_xue} and the references therein for more information on Makai's conjecture. Since the lattice normalized volume is a strictly monotonous functional, all the extremal cases in Makai's conjecture are necessarily lattice reduced.

Motivated by these analogies, we want to finish the paper by discussing discrete analogs of classical properties and conjectures about Euclidean reduced and complete convex bodies.

\subsection{Relations between diameter and width}
In the Euclidean setting, it is evident that the diameter is bounded from below by the width of a convex body, that is, we have $\diam_\RR(C)\geq \wdt_\RR(C)$ for any convex body $C\subset\RR^d$. The lattice analog of this fact is inequality \eqref{eq:wd_ineq} in Remark \ref{rem:zorn}. In general, it is impossible to reverse these inequalities, even at the cost of a constant depending on the dimension. However, if $C\subset\RR^2$ is a Euclidean reduced convex body in the plane, it was shown by Lassak \cite{lassak90} that
\begin{equation}
    \label{eq:lassak_diam}
    \diam_\RR(C)\leq \sqrt{2}\wdt_\RR(C).
\end{equation}
Also, if $C\subset\RR^d$ is a Euclidean complete convex body, we readily have $\diam_\RR(C) = \wdt_\RR(C)$, since $C$ is of constant width (recall that in the Euclidean setting, ``complete'' is stronger than ``reduced''). In the lattice setting in the plane, the following analogue to Lassak's statement holds.

\begin{proposition}
    \label{prop:reverse_wd_ineq}
    Let $C\subset\RR^2$ be lattice reduced or complete with respect to a lattice $\Lambda\subset\RR^2$. Then, $\diam_\Lambda(C)\leq\wdt_\Lambda(C)$. 
\end{proposition}

We note that equality holds if $C = \conv\{\pm e_1, \pm e_2\}$ (for lattice reduced bodies) or $C=[-1,1]^2$ (for lattice complete bodies).

\begin{proof}
    Let $C$ be lattice reduced with respect to $\Lambda$. Consider a diameter direction $v\in\Lambda\setminus\{0\}$ of $C$. Since $C$ is 2-dimensional it follows from Proposition \ref{prop:reduced_exposed} that there exist two independent width directions of $C$. So let $y\in\Lambda^\star$ be a width direction with $v\cdot y\neq 0$. Let $[a,a+\diam_\Lambda(C)v]\subset C$ be the diameter realizing segment in direction of $v$. We have
    \[
    \wdt_\Lambda(C) =\wdt(C;y) \geq |y\cdot(a+\diam_\Lambda(C)v -a)| = \diam_\Lambda(C)\,|y\cdot v| \geq \diam_\Lambda(C), 
    \]
    where we used that $|y\cdot v| \geq 1$. The argument for lattice complete bodies is analogous, the existence of a non-orthogonal pair of diameter and width directions is guaranteed by Proposition \ref{prop:endpoints}.
\end{proof}

Just like Lassak's inequality \eqref{eq:lassak_diam} does not generalize to higher dimensions (see \cite{lassak03}), it is impossible to extend Proposition \ref{prop:reverse_wd_ineq} to dimensions $d\geq 3$. As a counterexample one can consider the cube $C=[-1,1]^3$ together with the lattice $\Lambda = A_2 \oplus \alpha\ZZ\eins $, for any $\alpha \geq 1$. This is a lattice complete convex body (in fact, it is a special case of Proposition \ref{prop:lifting}, see also Figure \ref{fig:lifting}) with $\diam_\Lambda(C)=2$ and $\wdt_\Lambda(C) = 2/\alpha$. Since $\alpha\geq 1$ is arbitrary, this shows that there cannot be a constant $c>0$ such that $\diam_\Lambda(C)\leq c\cdot\wdt_\Lambda(C)$ holds for any 3-dimensional lattice complete convex body. For lattice reduced bodies, it suffices to consider $C^\star$ and $\Lambda^\star$, since $C$ is symmetric.

\subsection{Volumetric inequalities}

In Euclidean geometry, the isodiametric inequality states that the volume of a convex body $C\subset\RR^d$ of Euclidean diameter $d$ is bounded from above by the volume of a Euclidean ball of diameter $d$ \cite[Inequality (7.22)]{schneider}. Again, it is not possible to bound the volume of $C$ from below in terms of the diameter, if $C$ is arbitrary. For complete convex bodies in the plane, however, such a lower bound exists: For any Euclidean complete convex body $C \subset \RR^2$ of diameter $d$ we have \cite{Lebesgue14, Blaschke15}
\begin{equation}
    \label{eq:blaschkelebesgue}
    \vol(C) \geq \vol(U_d),
\end{equation}
where $U_d$ is a \emph{Releaux triangle} of diameter $d$, i.e., the intersection of three disks of radius $d$ centered at the vertices of a regular triangle of side length $d$. Generalizing \eqref{eq:blaschkelebesgue} to higher dimensions is an open problem, see for instance \cite{anciaux11,ccg96}.

For the Euclidean width, the dual statement to the isodiametric inequality is the P\'al-Kakeya problem, which asks for the optimal lower bound on $\vol(C)$ in terms of $\wdt_\RR(C)$. While it is not possible to bound $\vol(C)$ from above by $\wdt_\RR(C)$, Lassak conjectured a nice reverse version of P\'al's inequality in the plane \cite{lassak05}: If $C\subset\RR^2$ is a Euclidean reduced convex body with $\wdt_\RR(C)=w$, let $B_w$ be a Euclidean disk with width $w$. Do we have
\begin{equation}
    \label{eq:lassak_conjecture}
    \vol(C)\leq \vol(B_w)?
\end{equation}
Interestingly, the equality cases in \eqref{eq:lassak_conjecture} (if true) are not only the disks of width $w$, but also the ``quarter of a disk'' $\{ x\in \RR_{\geq 0}^2 \colon |x|\leq w\}$.

For lattice reduced and lattice complete convex bodies the volumetric analogs of the isodiametric inequality and the P\'al-Kakeya problem are Minkowski's first theorem and Makai's conjecture. Moreover,
in the plane we obtain inequalities similar to \eqref{eq:blaschkelebesgue} and \eqref{eq:lassak_conjecture} with comparatively little effort. Recall that $\Vol_\Lambda(C) = \vol(C)\det(\Lambda)^{-1}$ denotes the lattice normalized volume of $C$.

\begin{proposition}
    \label{prop:reverse_volume}
    Let $C\subset\RR^2$ be a convex body and let $\Lambda\subset\RR^2$ be a 2-dimensional lattice.
    \begin{enumerate}
        \item If $C$ is reduced with respect to $\Lambda$, then $\Vol_\Lambda(C) \leq \wdt_\Lambda(C)^2$.
        \item If $C$ is complete with respect to $\Lambda$, then $\Vol_\Lambda(C) \geq \tfrac 12 \diam_\Lambda(C)^2$.
    \end{enumerate}
    Both inequalities are sharp.
\end{proposition}

\begin{proof}
    Without loss of generality, let $\Lambda=\ZZ^2$ (cf. Lemma \ref{lemma:properties}). For (1), let $y_1,y_2\in\ZZ^2$ be two independent width directions of $C$. After applying a unimodular transformation and a translation, we obtain a copy of $C$ that is contained in the square $Q=[0,\wdt_{\ZZ^2}(C)]^2$. 

    For (2), let $S_1,S_2$ be two non-parallel lattice diameter realizing segments. Again, after applying a unimodular transformation, we can assume that $S_i$ is parallel to $e_i$. Let $P = \conv(S_1\cup S_2)\subseteq C$. If the segments $S_1$ and $S_2$ intersect, it is clear that $\vol(P) = \tfrac{1}{2}\diam_\Lambda(C)^2$. Otherwise, we have $\vol(P) > \tfrac{1}{2}\diam_\Lambda(C)^2$, as can be seen by first applying a Steiner symmetrization with respect to one of the affine hulls of the $S_i$ (see \cite[Remark 1.1]{henkhenzehernandez} for details). The inequality now follows from $P\subseteq Q$.

    In order to see that the inequalities are sharp, we consider the twisted squares $Q_x$, $x\in (0,1)$, from Example \ref{ex:dim2}. $Q_x$ is simultaneously lattice reduced and lattice complete and we have $\wdt(Q_x) = 2$ and $\lim_{x\nearrow 1} \vol(Q_x) = 4$, as well as $\lim_{x\searrow 0} \diam(Q_x) = 2$ and $\lim_{x\searrow 0} \vol(Q_x) = 2$. 
    \end{proof}

\subsection{Bodies that are both reduced and complete}

Recall that a convex body is complete in the Euclidean setting if and only if it is of constant width. Hence, the class of Euclidean complete bodies is (properly) contained in the class of Euclidean reduced bodies.

 On the discrete side we have seen that for a fixed lattice $\Lambda\subset\RR^d$ the class of lattice complete bodies and the class of lattice reduced bodies are neither contained in one another, nor dual to one another (at least not in an obvious way). Nonetheless, it would be interesting to gain a better understanding on how the notions of completeness and reducedness interact in this setting:

\begin{question}
\label{quest:rc}
What can be said about convex bodies $C\subset\RR^d$ that are simultaneously lattice reduced and lattice complete? Do there exist origin-symmmetric convex bodies that are reduced and complete for $d\geq 3$?
\end{question}

Below we list the convex bodies that are both lattice reduced and lattice complete that we know.
\begin{enumerate}
    \item The simplices $S_d\subset\RR^d$ from Proposition \ref{prop:terminal}.
    \item The quadrangles $Q_x\subset \RR^2$, $x\in (-1,1)$ and the hexagon $H$ from Example \ref{ex:dim2}.
    \item The triangles obtained in Proposition \ref{prop:triangle_classification}.
    \item The local maximizers of the lattice width $T_3\subset\RR^3$ and $T_4\subset\RR^4$ in dimensions 3 and 4 (cf.\ Remark \ref{rem:locals}).
    \item The regular simplex $\Delta$ from \eqref{eq:reg_simplex} with respect to the lattice $\Lambda$ as in \eqref{eq:sym_lattice}, where $\alpha =\beta = \tfrac 34$ and $\gamma = \tfrac 12$. This can be verified using the \texttt{sagemath} script from Remark \ref{rem:locals}. Note that since $\gamma = \tfrac 12$, this simplex has an exceptional diameter segment $I$ passing between two of its edges. 
\end{enumerate}
 All examples in this list are either 2-dimensional, or simplices. To start with, it would be interesting to see any lattice reduced and lattice complete polytope in dimension higher than 2 which is not a simplex.

\section*{Acknowledgements}
The authors wish to thank Martina Juhnke-Kubitzke, Thomas Kahle, Raman Sanyal and Christian Stump for organizing the Combinatorial CoworkSpace 2022, where this collaboration started. Many thanks go to Gennadiy Averkov, Bernardo Gonz\'{a}lez Merino, Martin Henk, Antonello Macchia, Stefan Kuhlmann and Francisco Santos, for helpful discussions about various sections of the paper. We also thank the anonymous referee for their helpful remarks which significantly improved the quality of the paper. Ansgar Freyer is partially supported by the Austrian Science Fund (FWF) Project
P34446-N. 

 \newcommand{\noop}[1]{}

\begin{appendices}
    \section{Existence of realizers of $\flt(d)$}
    \label{sec:realizers}
    Since we are not aware of an immediate argument or a reference in the literature that shows that the flatness constant is achieved by a convex body, we present a proof here in order to keep the paper self-contained.

    \begin{proposition}
    \label{prop:realizers}
        Let $\Lambda\subset\RR^d$ be a $d$-dimensional lattice. There exists a hollow (with respect to $\Lambda$) convex body $C\subset\RR^d$ with $\wdt_\Lambda(C)=\flt(d)$.
    \end{proposition}

    The proof requires the selection theorems of Blaschke and Mahler.

    \begin{theorem}[Blaschke]
        \label{thm:blaschke}
        Let $(C_i)_{i\in \NN}$ be a sequence of convex bodies in $\RR^d$ such that there exist convex bodies $K,L\subset\RR^d$ with $K\subseteq C_i\subseteq L$ for all $i\in\NN$. Then, $(C_i)_{i\in\NN}$ has a convergent subsequence (in the Hausdorff metric).
    \end{theorem}

    Recall from Section \ref{sec:background} that two bases $A,B\in\GL_d(\RR)$ of a $d$-dimensional lattice $\Lambda\subset\RR^d$ differ only by a right multiplication with a unimodular matrix, i.e., we have $B = AU$, for some $U\in\GL_d(\ZZ)$.
    Therefore the $d$-dimensional lattices in $\RR^d$ correspond one-to-one to the elements of $\GL_d(\RR)/\GL_d(\ZZ)$. This turns the set of $d$-dimensional lattices into a (metrizable) topological space by considering the quotient topology on $\GL_d(\RR)/\GL_d(\ZZ)$.
    
    \begin{theorem}[Mahler]
    \label{thm:mahler}
        Let $(\Lambda_i)_{i\in\NN}$ be a sequence of $d$-dimensional lattices in $\RR^d$ such that there exist constants $c_1,c_2>0$ such that $\det\Lambda_i < c_1$ and $\sukz_1(B_d,\Lambda_i) > c_2$, where $B_d$ denotes the Euclidean unit ball. Then the sequence $(\Lambda_i)_{i\in\NN}$ has a convergent subsequence.
    \end{theorem}

    \begin{proof}[Proof of Proposition \ref{prop:realizers}]
        Without loss of generality, let $\Lambda=\ZZ^d$. Let $(C_i)_{i\in\NN}$ be a sequence of hollow convex bodies such that $\lim_{i\to\infty} \wdt_{\ZZ^d}(C_i) = \flt(d)$. By John's theorem \cite[Theorem 10.12.2]{schneider} there exist vectors $t_i\in\RR^d$ and linear transformations $A_i\in\GL_d(\RR)$ such that
        \begin{equation}
            \label{eq:john}
            B_d\subseteq A_iC_i+t_i\subseteq d\, B_d.
        \end{equation}
        Let $\widetilde{C_i} = A_iC_i+t_i$ and $\Lambda_i=A_i\ZZ^d$. By Theorem \ref{thm:blaschke}, we can assume that that $(\widetilde{C_i})_{i\in\NN}$ is convergent to a convex body $C$. Assume that also the sequence $(\Lambda_i)_{i\in\NN}$ converges to a lattice $\Lambda$. Then, by the continuity of the successive minima,
        \[
        \wdt_\Lambda(C) = \lim_{i\to\infty} \wdt_{\Lambda_i}(\widetilde{C_i}) = \lim_{i\to\infty} \wdt_{\ZZ^d}(C_i) = \flt(d),
        \]
        which proves the claim.

        In order to show that $(\Lambda_i)_{i\in\NN}$ has a convergent subsequence, we show that the hypothesis of Theorem \ref{thm:mahler} are met. First, we note that by Minkowski's first theorem on successive minima combined with the reverse Blaschke-Santal\'o inequality, there exists a weak resolution of Makai's conjecture, i.e., a constant $c_1>0$ depending only on $d$ such that $\vol(C_i)\geq c_1\wdt_{\ZZ^d}(C_i)^d$ (see for instance \cite[Lemma 3.2]{polyslicing}). Thus,
        \begin{equation}
        \label{eq:determinant}
        \det \Lambda_i = |\det A_i| = \frac{\vol(\widetilde{C_i})}{\vol(C_i)} \leq \frac{d^d \vol(B_d)}{c_1 \wdt_{\ZZ^d}(C_i)^d},
        \end{equation}
        where we also used \eqref{eq:john} to estimate the numerator. It was shown in \cite{codenottisantos} that $\flt(d)\geq \flt(d-1) + 1$. Since $\wdt_{\ZZ^d}(C_i)$ converges to $\flt(d)$, we may assume \begin{equation}
            \label{eq:wdt_lower_bound}
            \wdt_{\ZZ^d}(C_i)\geq \flt(d-1)+\tfrac 12.
        \end{equation} So we obtain from \eqref{eq:determinant} that $\det\Lambda_i\leq c_2$, for a constant $c_2>0$ depending only on $d$.

        In order to treat the successive minima, we first use the inclusion \eqref{eq:john} that leads to the estimate
        \[
        \sukz_1(B_d,\Lambda_i) \geq \diam_{\Lambda_i}(\widetilde{C_i})^{-1} = \diam_{\ZZ^d}(C_i)^{-1}.
        \]
        We can finish the proof by showing that
        \begin{equation}
            \label{eq:c3}
            \diam_{\ZZ^d}(C_i) < \max\{ \flt(d)+1, 2(\flt(d-1) + 1)\} =:c_3.
        \end{equation}
        Towards a contradiction, assume that $\diam_{\ZZ^d}(C_i) \geq c_3$. Then there exists a segment $S=[a,a+c_3v]\subset C_i$, where $v\in\ZZ^d\setminus\{0\}$. Let $y\in\ZZ^d$ be a width direction of $C_i$. We have
        \[
        \wdt_{\ZZ^d}(C_i) = \wdt(C_i;y) \geq c_3 |v\cdot y|.
        \]
        Since $v\cdot y\in\ZZ$ and $c_3 > \flt(d)$ it follows that the width directions of $C_i$ are orthogonal to $v$. Let $K_i$ be the orthogonal projection of $C_i$ on $v^\perp$ and let $\Gamma$ be the lattice obtained by projecting $\ZZ^d$ on $v^\perp$. Then we have $\wdt_\Gamma(K_i) = \wdt_{\ZZ^d}(C_i)$. 

        While $K_i$ is in general not hollow, we show in the following that it is ``almost hollow''.  Let $p\in v^\perp$ be the point to which $S$ projects under the orthogonal projection $\pi_v:\RR^d\to v^\perp$. For a point $x\in K_i\setminus\{p\}$, let
        \[
        r(x) = \frac{\vol_1([p,x])}{\vol_1(C_i\cap ( p+\RR_{\geq 0}\cdot (x-p) )}.
        \]
        This is the Minkowski functional of $C_i$ when $p$ is interpreted as the origin. Let $x\in \Gamma \cap \inter(K_i)$. Since the lattice length of $S$ is at least $c_3$, by convexity, the lattice length of the fiber of $x$ under the projection $\pi$ is at least $c_3(1-r(x))$. On  the other hand, since $C_i$ is hollow, this fiber must not contain an interior lattice point of $\ZZ^d$. Its lattice length is therefore bounded by 1 from above. This gives $r(x) \geq 1-\tfrac{1}{c_3}$. This means that the convex body $K_i^\prime$ obtained by scaling $K_i$ by a factor $1-\tfrac{1}{c_3}$ around $p$ is hollow with respect to $\Gamma$ and so we have with \eqref{eq:wdt_lower_bound}:
        \[
            \flt(d-1) +\tfrac 12 \leq \wdt_{\ZZ^d}(C_i) = \wdt_\Gamma(K_i) = \frac{\wdt_\Gamma(K_i^\prime)}{1-\frac{1}{c_3}}\leq \frac{\flt(d-1)}{1-\frac{1}{c_3}}.
        \]
        After rearranging this, we see that $c_3\leq 2\flt(d-1) +1$, which contradicts the choice of $c_3$ in \eqref{eq:c3}. This finishes the proof.
    \end{proof}

\end{appendices}

\end{document}